\documentclass[12pt]{amsart}
\usepackage{amssymb,amsmath, amsthm,latexsym,verbatim}
\usepackage{a4wide}
\usepackage[hypertex]{hyperref}
\usepackage{enumitem}

\theoremstyle{plain}

\newtheorem{lem}{Lemma}[section]
\newtheorem{theo}[lem]{Theorem}
\newtheorem{prop}[lem]{Proposition}
\newtheorem{corollary}[lem]{Corollary}

\parskip=\medskipamount
\font\k=cmr7

  \newcommand {\pr}{\mbox{\k prime}}

  \newcommand {\sym}{\mbox{\k sym}}
  \newcommand {\C}{{\mathbb C}}
  \newcommand {\bH}{{\mathbb H}}
  \newcommand {\N}{{\mathbb N}}
  \newcommand {\R}{{\mathbb R}}
  \newcommand {\Z}{{\mathbb Z}}
  \newcommand {\Q}{{\mathbb Q}}

  \newcommand {\af}{{\mathfrak a}}
  \newcommand {\gf}{{\mathfrak g}}
  \newcommand {\mf}{{\mathfrak m}}
  \newcommand {\kf}{{\mathfrak k}}

  \newcommand {\hf}{{\mathfrak h}}
  
  \newcommand {\nf}{{\mathfrak n}}
  
  \newcommand {\pg}{{\mathfrak p}}

\renewcommand {\H}{{\mathcal H}}
  
  \newcommand {\Co}{{\mathcal C}}

  \newcommand {\E}{{\mathcal E}}

\newcommand  {\cZ}{{\mathcal Z}}
\newcommand {\bs}{\backslash}

\renewcommand{\Re}{\operatorname{Re}}
\newcommand{\Tr}{\operatorname{Tr}}
\newcommand{\Spec}{\operatorname{Spec}}
\newcommand{\End}{\operatorname{End}}

\newcommand{\tr}{\operatorname{tr}}
\newcommand{\Id}{\operatorname{Id}}

\newcommand{\Ind}{\operatorname{Ind}}
\newcommand{\ch}{\operatorname{ch}}
\newcommand{\I}{\operatorname{I}}

\newcommand{\vol}{\operatorname{vol}}

\newcommand{\SL}{\operatorname{SL}}
\newcommand{\GL}{\operatorname{GL}}
\newcommand{\SO}{\operatorname{SO}}
\newcommand{\SU}{\operatorname{SU}}

\newcommand{\Spin}{\operatorname{Spin}}
\newcommand{\Ad}{\operatorname{Ad}}

\renewcommand{\det}{\operatorname{det}}
\newcommand{\gr}{\mathrm{gr}}
\newcommand{\Sym}{\operatorname{Sym}}
\def\OP#1{\operatorname{#1}}

\begin{document}
%Topmatter
\title[]
{The asymptotics of the Ray-Singer analytic torsion of hyperbolic 3-manifolds}
%\date{\today}

\author{Werner M\"uller}
\address{Universit\"at Bonn\\
Mathematisches Institut\\
Beringstrasse 1\\
D -- 53115 Bonn, Germany}
\email{mueller@math.uni-bonn.de}

\begin{abstract}
In this paper we consider the analytic torsion of a closed hyperbolic
3-manifold associated with the $m$-th symmetric power 
of the standard representation of $\SL(2,\C)$   and
we study its asymptotic behavior if $m$ tends to infinity. The leading 
coefficient of the asymptotic formula is given by the volume of the 
hyperbolic 3-manifold. It follows that the Reidemeister torsion
 associated with the
symmetric powers determines the volume of a closed hyperbolic 3-manifold.
\end{abstract}

\maketitle
\setcounter{tocdepth}{1}
%\tableofcontents
\section{Introduction}

Let $X$ be a closed, oriented hyperbolic 3-manifold. Then there
exists a discrete, torsion free, co-compact  subgroup 
$\Gamma\subset \SL(2,\C)$ such that $X=\Gamma\bs\bH^3$, where 
$\bH^3=\SL(2,\C)/\SU(2)$ is the 3-dimensional hyperbolic space. Let $\rho$ be
a finite-dimensional complex representation of $\Gamma$ and let $E_\rho\to X$
be the associated flat vector bundle. Choose a Hermitian fiber metric $h$ on
$E_\rho$. Let $T_X(\rho;g,h)$ denote the Ray-Singer analytic torsion of the
de Rham complex of $E_\rho$-valued differential forms \cite{RS}, 
where $g$ denotes the hyperbolic metric. If $\rho$ is acyclic, then 
$T_X(\rho;g,h)$ is metric independent \cite[Corollary 2.7]{Mu1}. In this
case we denote it by $T_X(\rho)$. 

For $m\in\N$ let $\tau_m=\Sym^m$ be the $m$-th symmetric power of the 
standard representation of $\SL(2,\C)$ on $\C^2$ and denote by $E_{\tau_m}$ the 
flat vector bundle associated to $\tau_m|_\Gamma$. It is well known that 
$H^*(X,E_{\tau_m})=0$. This follows, for example, from
\cite[Chapt. VII, Theorem 6.7]{BW}. Hence the restriction of 
$\tau_m$ to $\Gamma$ is an acyclic representation of $\Gamma$. Denote by 
$T_X(\tau_m)$ the analytic torsion with respect to $\tau_m|_\Gamma$. The  
purpose of this paper is to study the asymptotic behavior of $T_X(\tau_m)$ 
as $m\to\infty$. Our main result is the following theorem.
\begin{theo}\label{asymptotic}
Let $X$ be a closed, oriented hyperbolic 3-manifold. Then we have
\begin{equation}\label{asymp0}
-\log T_X(\tau_m)=\frac{\vol(X)}{4\pi}m^2+O(m)
\end{equation}
as $m\to\infty$.
\end{theo} 
We note that there is an analogous result in the holomorphic setting.
In \cite{BV}  Bismut and Vasserot  studied the asymptotic
behavior of the holomorphic Ray-Singer torsion for symmetric powers
of a positive vector bundle.  

Let $\tau_X(\tau_m)$ denote the 
Reidemeister torsion of $X$ with respect to $\tau_m|_\Gamma$ (see \cite{Mu1}).
Then by \cite[Theorem 1]{Mu1} we have $T_X(\tau_m)=\tau_X(\tau_m)$. Thus we 
obtain the following corollary.
\begin{corollary}\label{asymptotic1}
Let $X$ be a closed, oriented hyperbolic 3-manifold. Then we have
\begin{equation}\label{asymp1}
-\log \tau_X(\tau_m)=\frac{\vol(X)}{4\pi}m^2+O(m)
\end{equation}
as $m\to\infty$.
\end{corollary}
This result has applications to the cohomology of arithmetic hyperbolic
3-manifolds. We will discuss this elsewhere.
As an immediate corollary we get
\begin{corollary}\label{volume}
Let $X$ be a closed, oriented hyperbolic 3-manifold. Then $\vol(X)$ is
determined by the set of Reidemeister torsions 
$\{\tau_X(\tau_m)\colon m\in\N\}$.
\end{corollary}
Some remarks are in order. The Reidemeister torsion of a compact 3-manifold
is known to be a topological invariant \cite{Ch}. Therefore, it follows from 
Corollary \ref{volume} that the volume of a compact, oriented 
hyperbolic 3-manifold is a topological invariant. This is also
a well known consequence of the Mostow-Prasad rigidity theorem \cite{Mo,Pr}.

There are only finitely many closed, oriented hyperbolic 3-manifolds with the
same volume \cite[Theorem 3.6]{T1}. Therefore we get
\begin{corollary}
A compact, oriented hyperbolic 3-manifold $X$ is determined up to finitely 
many possibilities by the set $\{\tau_X(\tau_m)\colon m\in\N\}$
 of Reidemeister torsion invariants. 
\end{corollary}
It is known \cite{Zi} that the number of closed hyperbolic manifolds with a
given volume can be arbitrarily large. Therefore the proof of the corollary
does not give a uniform bound on the number of closed hyperbolic manifolds with
the same set of Reidemeister torsion invariants.

Our approach to prove Theorem \ref{asymptotic} is based on the expression of 
$T_X(\tau_m)$ in
terms of the twisted Ruelle zeta function attached to $\tau_m$. Recall that
for a finite-dimensional complex representation $\rho$ of $\Gamma$ the twisted 
Ruelle zeta function $R_\rho(s)$ is defined for $\Re(s)\gg 0$ as the infinite
product
\begin{equation}
R_\rho(s)=\prod_{\substack{[\gamma]\neq e\\\pr}}
\det\left(\Id-\rho(\gamma)e^{-s\ell(\gamma)}\right),
\end{equation}
where $[\gamma]$ runs over the nontrivial primitive conjugacy classes of
$\Gamma$ and $\ell(\gamma)$ denotes the length of the unique closed
geodesic associated to $[\gamma]$. It follows from \cite{Fr2} that 
$R_\rho(s)$ admits a meromorphic
extension to the whole complex plane. If $\rho$ is unitary and acyclic
then  $R_\rho(s)$ is regular at $s=0$ and its value at zero satisfies
\[
|R_\rho(0)|=T_X(\rho)^2
\] 
(see \cite{Fr}). For an arbitrary unitary representation (which is not
necessarily acyclic), the coefficient of the leading term of the 
Laurent expansion of 
$R_\rho(s)$ at $s=0$ is given by the analytic torsion. The corresponding 
result holds for any compact, oriented hyperbolic manifold of dimension $n$ 
\cite{Fr}. In his thesis \cite{Br} U. Br\"ocker has established a similar 
result for representations of the fundamental group that are restrictions of 
finite-dimensional irreducible representations of the isometry group 
$\SO_0(n,1)$ of the hyperbolic $n$-space. Unfortunately, his method is
based on elaborate computations which are difficult to verify. This problem
has been rectified by Wotzke in his thesis \cite{Wo}. He gave a different proof
which replaces Br\"ocker's explicite computations by the real version of 
Kostant's Bott-Borel-Weil theorem \cite{Si}.

To state the result for $n=3$ we need to introduce some notation. 
Let $\tau$ be a finite-dimensional, irreducible representation of $\SL(2,\C)$,
which we regard as real Lie group. Let $\theta$ be the Cartan involution of 
$\SL(2,\C)$ with respect to $SU(2)$. Put $\tau_\theta=\tau\circ\theta$. 
Denote by $R_\tau(s)$ the twisted Ruelle zeta function for the restriction of 
$\tau$ to $\Gamma$. Let $E_\tau\to X$ be the flat vector bundle associated to
$\tau|_\Gamma$.  The flat bundle $E_\tau$ can be equipped with a canonical 
Hermitian fiber metric \cite{MM}. Let $\Delta_p(\tau)$ be the corresponding
Laplacian on $E_\tau$-valued $p$-forms and denote by $T_X(\tau)$ the Ray-Singer
analytic torsion associated to $\tau|_\Gamma$. Then the main result of 
\cite{Wo} for $n=3$ is the following theorem.
\begin{theo}\label{theo-ruelle}
Let $\tau$ be a  finite-dimensional, irreducible representation of 
the real Lie group $\SL(2,\C)$. Then we have\\
$\mathrm{1)}$ If $\tau_\theta\ncong\tau$, then $R_\tau(s)$ is regular at $s=0$ 
and
\[
|R_\tau(0)|=T_X(\tau)^2.
\]
$\mathrm{2)}$ Let $\tau_\theta=\tau$. If $\tau\neq 1$, then the order 
$h(\tau)$ of $R_\tau(s)$ at $s=0$ is given by 
\begin{equation}\label{order5}
h(\tau)=2\sum_{p=1}^3(-1)^p\dim\ker\Delta_p(\tau).
\end{equation}
and if $\tau=1$, the order equals $4-2\dim H^1(X,\R)$.
The leading term of
the Laurent expansion of $R_\tau(s)$ at $s=0$ is given by
\[
T_X(\tau)^2 s^{h(\tau)}.
\]
\end{theo}
The case of the trivial representation is covered by \cite{Fr}. In this case 
the order of $R(s)$ at $s=0$ differs from \eqref{order5}.

It follows from Theorem \ref{theo-ruelle} that in order to prove Theorem 
\ref{asymptotic} it suffices to analyze the asymptotic behavior of 
$R_{\tau_m}(0)$ as $m\to\infty$. For this purpose we consider another type of 
twisted Ruelle 
zeta functions. Let $A$ be the standard split torus of $\SL(2,\C)$ and let
$M$ be the centralizer of $A$ in $\SU(2)$ (see \eqref{tori} for the explicit
description). For $\sigma\in\hat M$ let $R(s,\sigma)$ be the Ruelle zeta
function defined by \eqref{ruellezeta}. Using the decomposition of $\tau_m$ 
under the Cartan subgroup $MA$, it follows that $R_{\tau_m}(s)$ is the product 
of the twisted Ruelle zeta functions with shifted argument 
$R(s-(m/2-k),\sigma_{m-2k})$, $k=0,...,m$. This reduces the study of 
the  asymptotic behavior 
of $T_X(\tau_{2m})$ (resp. $T_X(\tau_{2m+1})$) as $m\to\infty$ to the study of 
the behavior of $|R(k,\sigma_{2k})|$ and $|R(-k,\sigma_{2k})|$, 
(resp. $|R(k+1/2,\sigma_{2k+1})|$ and 
$|R(-k-1/2,\sigma_{-(2k+1})|$), $k>2$, as $k\to\infty$. 
To analyze the behavior of $|R(k,\sigma_{2k})|$ (resp. 
$|R(k+1/2,\sigma_{2k+1})|$) as $k\to\infty$ we simply
use the infinite product defining it. To deal with the remaining cases we use
the functional equation which implies
\[
|R(-s,\sigma_k)|=\exp\left(-4\pi^{-1}\vol(X)\Re(s)\right)|R(s,\sigma_{-k})|.
\]
This is exactly how the volume of $X$ appears in the asymptotic formula 
(\ref{asymp0}). 

For the sake of completeness we include a proof of Theorem \ref{theo-ruelle}
which is based on results of \cite{BO}. The starting point of the method is the
observation that the flat bundle $E_\tau\to X$ is isomorphic to the locally
homogeneous vector bundle defined by the restriction of $\tau$ to the maximal
cocompact subgroup $\SU(2)$ of $\SL(2,\C)$ (see \cite[Propostion 3.1]{MM}).
Using this isomorphism the bundle $E_\tau$ can be equipped with a canonical
Hermitian fiber metric induced from an invariant  metric on the corresponding
homogeneous vector bundle \cite[Lemma 3.1]{MM}. We define the Laplacian 
$\Delta_p(\tau)$ on $E_\tau$-valued $p$-forms with respect to this metric on
$E_\tau$. Then, up to a constant, $\Delta_p(\tau)$ equals $-R(\Omega)$, where
$R(\Omega)$ denotes the action of the Casimir operator on sections of the 
locally homogeneous bundle. This is the key fact which allows us to apply
the Selberg trace formula to the heat kernel.

The paper is organized as follows. In section \ref{prelim} we summerize some 
basic facts
about hyperbolic 3-manifolds and analytic torsion. In section \ref{sec-ruelle}
we consider twisted Ruelle and Selberg zeta functions and establish
some of their basic properties. In section \ref{sec-bochlapl}
we introduce certain auxiliary elliptic operators which are needed to derive
the determinant formula and to prove the functional equation for the Selberg
zeta function. In the next section \ref{sect-funcequ} we establish the 
functional equation for the Selberg and Ruelle zeta functions. 
In section \ref{detformula} we use the determinant formula of
\cite{BO} to express the twisted Ruelle zeta function as a ratio of products
of regularized determinants of the  elliptic operators introduced in section
\ref{sect-funcequ}. The determinant formula is one of main tools to study the 
leading term of the
Laurent expansion of $R_\tau(s)$ at $s=0$. In section \ref{heattrace} we give 
a proof of Theorem \ref{theo-ruelle}. In the final section \ref{mainth} 
we proof Theorem \ref{asymptotic}.

\subsection*{Acknowledgement} I would like to thank Jonathan Pfaff for a
careful reading of the manuscript and for pointing out some minor mistakes.

\section{Preliminaries}\label{prelim}
\setcounter{equation}{0}

\subsection{}

Let $G=\SL(2,\C)$ and $K=\SU(2)$. Then $K$ is a maximal compact subgroup of 
$G$. We regard $G$ as real Lie group and we recall that $G$ is isomorphic to
$\Spin_0(3,1)$ \cite{Be}. Under this isomorphism, $\SU(2)$ is mapped
to $\Spin(3)$. Thus $G$ acts on the hyperbolic 3-space $\bH^3$ and 
$\bH^3\cong G/K$. Let $\Gamma\subset G$ be a cocompact, torsion free discrete
subgroup. Then $X=\Gamma\bs \bH^3$ is a compact, oriented hyperbolic 3-manifold
and any such manifold is of this form.  

Let $G=NAK$ be the standard Iwasawa decomposition of $G$ and
let $M$ be the centralizer of $A$ in $K$. Then
\begin{equation}\label{tori}
A=\left\{\begin{pmatrix}\lambda&0\\ 0&\lambda^{-1}\end{pmatrix}\colon\lambda\in
\R^+\right\},\quad M=\left\{\begin{pmatrix} e^{i\theta}&0\\ 0& e^{-i\theta}
\end{pmatrix}\colon\theta\in [0,2\pi]\right\}.
\end{equation}
We use the natural normalization of the Haar measures for $A,N,K$ and $G$ as in 
\cite[pp. 387-388]{Kn}. In particular, we choose on $K$ the Haar measure $dk$
of total mass $1$. 

Let $\gf$, $\kf$, $\af$, $\mf$ and $\nf$ denote the Lie algebras of 
$G$, $K$, $A$, $M$  and $N$, respectively.  Let
\begin{equation}\label{cartan}
\gf=\kf\oplus \pg
\end{equation}
be the Cartan decomposition of $(\gf,\af)$. Then $\af$ is a maximal abelian 
subspace of $\pg$. Let $\alpha$ be the unique positive root of $(\gf,\af)$. 
Let $H\in\af$ be such that $\alpha(H)=1$. Let $\af^+\subset\af$ be 
the positive Weyl chamber and let $A^+=\exp(\af^+)$. Let $W:=W(\gf,\af)$ 
denote the restricted Weyl group.

Put $\hf=\mf\oplus\af$. Then $\hf$ is a Cartan subalgebra 
of $\gf$. We identify $\hf$ with $\R^2$. Then the
Weyl group $W_G$ of $(\gf_\C,\hf_\C)$ acts on $\C^2$ by 
sign changes. So $W_G$ has order $4$. In a compatible 
ordering on $\hf^*_\C$ the only positive roots of the pair $(\gf_\C,\hf_\C)$
are $\alpha_1$ and $\alpha_2$ where $\alpha_1(H)=\alpha_2(H)=\alpha(H)=1$
and  $\alpha_1(iH)=-\alpha_2(iH)=i$. Let $\rho_G=\frac{1}{2}
(\alpha_1+\alpha_2)$. 

Let $B$ be the Killing form of $\gf$. Define a symmetric bilinear form on $\gf$
by
\begin{equation}\label{biform}
\langle Y_1,Y_2\rangle =\frac{1}{4}B(Y_1,Y_2),\quad Y_1,Y_2\in\gf.
\end{equation}  
Then $\langle\cdot,\cdot\rangle$ is positive definite on $\pg$, negative 
definite on $\kf$ and we have
$\langle\kf,\pg\rangle=0$. The normalization is such that the restriction of 
$\langle\cdot,\cdot\rangle$ to $\pg\cong T_{eK}(G/K)$ induces the $G$-invariant
Riemannian metric on $\bH^3=G/K$ which has constant curvature $-1$.

Let $\{Z_i\}$ be a basis of $\gf$ and let $\{Z^j\}$ be the basis of $\gf$
which is determined by $\langle Z_i,Z^j\rangle=\delta_{ij}$. Then the Casimir
element $\Omega\in\cZ(\gf_\C)$ is given by
\begin{equation}
\Omega=\sum_i Z_iZ^i.
\end{equation}
Let $R(\Omega)$ be the differential operator induced by $\Omega$ on 
$C^\infty(\bH^3)$. Then by Kuga's lemma we have
\[
R(\Omega)=-\Delta,
\]
where $\Delta$ is the hyperbolic Laplace operator on functions.

\subsection{}

Let $\Gamma\subset G$ be a discrete, torsion free, cocompact subgroup. Then
$X=\Gamma\bs\bH^3$ is a closed hyperbolic manifold. Given $\gamma\in \Gamma$,
we denote by $[\gamma]$ the $\Gamma$-conjugacy class of $\gamma$. The set of 
all
conjugacy classes of $\Gamma$ will be denoted by $C(\Gamma)$. Let 
$\gamma\not=1$. Then there exist $g\in G$, $m_\gamma\in M$, and 
$a_\gamma\in A^+$ such that
\begin{equation}\label{conjug}
g\gamma g^{-1}=m_\gamma a_\gamma.
\end{equation}
By \cite[Lemma 6.6]{Wa1}, $a_\gamma$ depends only on $\gamma$ and $m_\gamma$ 
is determined up to conjugacy in $M$.  By definition
there exists $\ell(\gamma)>0$  such that
\begin{equation}\label{length}
a_\gamma=\exp\left(\ell(\gamma)H\right).
\end{equation}
Then $\ell(\gamma)$ is the length of the unique closed geodesic in $X$ 
that corresponds to the conjugacy class $\{\gamma\}$. An element 
$\gamma\in\Gamma-\{e\}$ is called primitive, if it can not be written as 
$\gamma=\gamma_0^k$ for some $\gamma_0\in\Gamma$ and $k>1$. For every 
$\gamma \in\Gamma-\{e\}$ there exist a unique primitive element $\gamma_0\in
\Gamma$ and $n_\Gamma(\gamma)\in\N$ such that $\gamma=\gamma_0^{n_\Gamma(\gamma)}$.

\subsection{}

Denote by $\hat M$ the set of unitary characters of $M$. Then $\hat M\cong\Z$
and the character $\sigma_k$ that corresponds to $k\in\Z$ is given by
\begin{equation}\label{uncharac}
\sigma_k\left(\begin{pmatrix} e^{i\theta}&0\\ 0& e^{-i\theta}
\end{pmatrix}\right)=e^{ik\theta}.
\end{equation}
The finite-dimensional irreducible representations of $G$, regarded as real 
Lie group, are parametrized by pairs of nonnegative integers \cite[p. 32]{Kn}. 
For $m\in\N_0$
let 
\[
\tau_m=\Sym^m\colon G\to\GL(S^m(\C^2))
\]
be the $m$-th symmetric power of the standard representation of $G=\SL(2,\C)$
on $\C^2$. Denote by $\overline \tau_m$ the complex conjugate representation.
Then
\begin{equation}\label{sl2irred}
\tau_{m,n}=\tau_m\otimes\overline\tau_n
\end{equation}
is the irreducible representation with highest weight $(m,n)$.  
The restrictions of $\tau_m$  to $MA$ decomposes as follows:
\begin{equation}\label{restrict5}
\tau_m\big|_{MA}=\bigoplus_{k=0}^m\sigma_{m-2k}\otimes e^{(\frac{m}{2}-k)\alpha}.
\end{equation}

\subsection{}

Let $P=MAN$ be the standard parabolic subgroup of $G$. We identify $\C$ with
$\af_\C^*$ by $z\to zH$. 
For $n\in\Z$ and $\lambda\in\C$ let $\pi_{n,\lambda}$ be the induced 
representation
\begin{equation}
\pi_{n,\lambda}=\Ind_P^G(\sigma_n\otimes e^{i\lambda}\otimes 1). 
\end{equation}
Note that this is the parametrization of the principal series used in 
\cite[Chapt. XI, \S2]{Kn}. The representation $\pi_{n,\lambda}$  
acts in the Hilbert space $\H_{n,\lambda}$ whose subspace of 
$C^\infty$-vectors is given by
\begin{equation}
\H^\infty_{n,\lambda}=\left\{f\in C^\infty(G,V_\sigma)\colon f(gman)=
e^{-(i\lambda+1)(\log a)}\sigma(m)^{-1}f(g),\quad g\in G,\;man\in P\right\}.
\end{equation}
If $\lambda\in\R$, then $\pi_{n,\lambda}$ is unitary. This family of 
representations is the unitary principal series. All $\pi_{n,\lambda}$, $n\in\Z$,
$\lambda\in\R\setminus\{0\}$, are irreducible. They have an  
explicit realization \cite[Chapt. II, \S4]{Kn}. The Casimir eigenvalue 
$\pi_{n,\lambda}(\Omega)$ is given by
\begin{equation}\label{ucasimir}
\pi_{n,\lambda}(\Omega)=-\lambda^2+\frac{n^2}{4}-1.
\end{equation} 
This follows from \cite[Theorem 8.22]{Kn}. It can be also verified using the 
explicit realization of $\pi_{n,\lambda}$. In the latter case one has to take
into account that the identification of $\af_\C$ with $\C$ is differnt from 
ours.

The nonunitarily induced representations
\begin{equation}
\pi_x^c=\Ind_P^G(1\otimes e^{x}\otimes 1),\quad 0<x<2,
\end{equation}
are unitarizable. This is the complementary series. The Casimir eigenvalue is
given by
\begin{equation}\label{casimir8}
\pi_x^c(\Omega)=x^2-1.
\end{equation}
This also follows from \cite[Theorem 8.22]{Kn} or can be verified using the 
explicit realization of $\pi_x^c$ (see \cite[Chapt. II, \S4]{Kn}). Denote by
$\Theta_{n,\lambda}=\tr\pi_{n,\lambda}$ the character of $\pi_{n,\lambda}$.
 
\subsection{}

Let $\tau\colon G\to \GL(V_\tau)$ be an irreducible finite-dimensional 
representation of $G$. 
Let  $E_{\tau}$ be the flat vector bundle associated to the
restriction $\tau|_\Gamma$ of $\tau$ to $\Gamma$. By 
\cite[Proposition 3.1]{MM} $E_\tau$ is canonically isomorphic to the locally 
homogeneous vector bundle associated to $\tau|_K$, i.e.,
\begin{equation}\label{iso}
E_{\tau}\cong (\Gamma\bs G\times V_\tau)/K,
\end{equation}
where $K$ acts on $\Gamma\bs G\times V_\tau$ by $(\Gamma g,v)\cdot k=
(\Gamma gk,\tau(k)^{-1}v)$. 
So we may regard $E_\tau$ as locally homogeneous vector bundle equipped with
a flat connection which, of course, is different from the canonical invariant
connection on the homogeneous bundle.

The vector bundle $E_\tau$ can be equipped with a canonical fiber metric. By 
\cite[Lemma 3.1]{MM} there exists an inner product $\langle\cdot,\cdot\rangle$
on $V_\tau$ which satisfies 
\begin{equation}
\begin{split}
(a)\qquad \langle\tau(Y)u,v\rangle&=-\langle u,\tau(Y)v\rangle\text{ for all
$Y\in\mathfrak k$, $u,v\in V_\tau$;}\\
(b)\qquad \langle\tau(Y)u,v\rangle&=\langle u,\tau(Y)v\rangle
\mspace{13mu}\text{ for all $Y\in\mathfrak p$, $u,v\in V_\tau$.} 
\end{split}
\end{equation}
Such an inner product is called {\it admissible}. It is unique up to scaling.
By (a) the inner product is invariant under $\tau(K)$ and therefore, it defines
via \eqref{iso} a Hermitian fiber metric in $E_\tau$. Denote by $\Delta_p(\tau)$
the Laplacian on $E_\tau$-valued $p$-forms with respect to an admissible metric
on $E_\tau$. 

\subsection{}

Let $P$ be an elliptic differential operator acting on 
$C^\infty$-sections of a smooth Hermitian vector bundle $E$ over a compact 
Riemannian manifold $X$. The metrics $g$ on $X$  and $h$ on $E$ induce an 
inner product 
in $C^\infty(X,E)$. Suppose that with respect to this inner product 
the operator $P$ is symmetric and nonnegative. Then the zeta function 
$\zeta(s;P)$, $s\in\C$, of $P$ is defined as
\[
\zeta(s;P)=\sum_{\lambda\in\Spec(P)\setminus\{0\}} m(\lambda)\lambda^{-s},
\]
where $m(\lambda)$ denotes the multiplicity of the eigenvalue $\lambda$.
The series converges absolutely and uniformly on compact subsets of 
$\Re(s)>\dim(X)/\mathrm{ord}(P)$. Moreover $\zeta(s;P)$ admits a meromorphic 
extension to $s\in\C$ which is holomorphic at $s=0$ 
(see \cite[Chapt. II]{Sh}). Then the regularized determinant $\det P$ of $P$ 
is defined as
\begin{equation}\label{regdet}
\det P=\exp\left(-\frac{d}{ds}\zeta(s;P)\big|_{s=0}\right).
\end{equation}
Assume that $P$ is symmetric and bounded from
below. Let $\lambda\in\R$ be such that
$P+\lambda>0$. Then $\det(P+\lambda)$ is defined by \eqref{regdet}. 
Voros \cite{Vo} has shown that the function 
$\lambda\mapsto \det(P+\lambda)$, defined for $\lambda\gg0$,  extends to
an entire function $\det(P+s)$ of $s\in\C$  with zeros $-\lambda_j$ where
$\lambda_j\in\Spec(P)$.

\subsection{}

Finally we recall the definition of the Ray-Singer analytic torsion \cite{RS}.
Let $\chi$ be a finite-dimensional representation 
of $\pi_1(X)$ and let $E_\chi\to X$ be the associated flat vector bundle over 
$X$. Pick a Hermitian fiber 
metric $h$ in $E_\chi$ and let $\Delta_{p}(\chi)\colon
\Lambda^p(X,E_\chi)\to \Lambda^p(X,E_\chi)$ be the Laplacian on the space of 
$E_\chi$-valued $p$-forms. Then $\Delta_{p}(\chi)$ is a
nonnegative,  second order elliptic differential operator. So it has a well
defined  regularized determinant, defined by \eqref{regdet}. Then the
analytic torsion is defined as the following weighted product of regularized 
determinants
\begin{equation}\label{anator1}
T_X(\chi;g,h)=\prod_{p=1}^3\left(\det \Delta_p(\chi)\right)^{(-1)^{p+1}p/2}.
\end{equation}
By definition $T_X(\chi;g,h)$ depends on $g$ and $h$. However, if $\dim X$ is
odd and $\chi$ is acyclic, 
i.e., $H^*(X,E_\chi)=0$, then $T_X(\chi;g,h)$ is independent of $g$ and $h$ 
\cite[Corollary 2.7]{Mu1}. In this case we denote it simply by $T_X(\chi)$.

In this paper we consider the special case where $X=\Gamma\bs\bH^3$ is a closed
hyperbolic 
3-manifold and $\chi$ is the restriction of a representation 
$\tau\colon G\to \GL(V_\tau)$ to $\Gamma$. Then, as explained above,
 the flat bundle $E_\tau$ carries an admissible metric. 
We denote the analytic torsion attached to $\tau|_\Gamma$ with respect to this
metric by $T_X(\tau)$.

\section{Twisted Ruelle and Selberg zeta functions}\label{sec-ruelle}
\setcounter{equation}{0}

In this section we consider various kinds of twisted geometric zeta functions
which are needed for the proof of our main result. We will use the notation
introduced in section \ref{prelim}.
First we recall the following estimation of the growth of the length spectrum.
For $R>0$  we have
\begin{equation}\label{growth}
\#\big\{[\gamma]\in C(\Gamma)\colon \ell(\gamma)\le R\big\}\ll e^{2R}
\end{equation}
\cite[(1.31)]{BO}.

If $T$ is an endomorphism of a finite-dimensional vector
space, denote by $S^kT$ the $k$-th symmetric power of $T$. Let 
$\overline\nf=\theta\nf$ be the negative root space. Then for 
$\sigma\in\hat M$ and $s\in\C$ 
with $\Re(s)>2$ the twisted Selberg zeta function is defined by
\begin{equation}\label{selzeta}
Z(s,\sigma)=\prod_{\substack{[\gamma]\not=e\\\pr}}\prod_{k=0}^\infty
\det\left(1-\left(\sigma(m_\gamma)\otimes S^k
\left(\Ad(m_\gamma a_\gamma)_{\overline\nf}\right)\right)
e^{-(s+1)\ell(\gamma)}\right),
\end{equation}
where $[\gamma]$ runs over the non-trivial primitive conjugacy classes in 
$\Gamma$. By \cite[(3.6)]{BO} we have
\begin{equation}\label{logzeta}
\log Z(s,\sigma)=-\sum_{[\gamma]\neq e}\frac{\sigma(m_\gamma)e^{-\ell(\gamma)}}
{\det(\Id-\Ad(m_\gamma a_\gamma)_{\overline{\nf}})\mu(\gamma)}\,e^{-s\ell(\gamma)}.
\end{equation}
It follows from
\eqref{growth} that the series converges absolutely and
uniformly in the half-plane $\Re(s)>2$. Therefore the infinite product 
converges absolutely and
uniformly in the half-plane $\Re(s)>2$.  Furthermore
by \cite[Theorem 3.15]{BO} it has a meromorphic extension to the entire 
complex plane and satisfies a functional equation \cite[Theorem 3.18]{BO}. 
To state the functional equation we need some notation.  Let $w\in W_A$ be the
non-trivial element. Then $w$ acts on $\hat M$ by
\[
(w\sigma)(m)=\sigma(m_w^{-1}m m_w),\quad m\in M,\, \sigma\in\hat M,
\]
where $m_w$ is a representative of $w$ in the normalizer of $\af$ in $K$. Thus
$w\sigma_k=\sigma_{-k}$, $k\in\Z$. For each $\sigma\in\hat M$ there is an
associated Dirac operator $D_X(\sigma)$ acting in a Clifford bundle 
$E_\sigma\to X$ \cite[p.29]{BO}. Let $\eta(D_X(\sigma))$ denote the eta 
invariant of $D_X(\sigma)$. Let $P_\sigma$ be the Plancherel polynomial with
respect to $\sigma$. If the Haar measures are normalized as in 
\cite[pp. 387-388]{Kn} and $\af_\C$ is identified with $\C$ by $z\in\C\mapsto
zH\in\af_\C$, then by \cite[Theorem 11.8]{Kn} 
(up to a minor correction) it is given by 
\begin{equation}\label{planch2}
P_{\sigma_k}(z)=\frac{1}{4\pi^2}\left(\frac{k^2}{4}-z^2\right),\quad k\in\Z,
\end{equation}
where $\sigma_k\in\hat M$ is the character defined by \eqref{uncharac}. 
We note that our definition of $P_\sigma(z)$ differs from the definition of 
$P_\sigma(z)$ in \cite[p. 56]{BO}.

Then the functional equation satisfied by 
$Z(s,\sigma)$ is the following equality
\begin{equation}\label{functequ1}
Z(s,\sigma)=e^{i\pi\eta(D_X(\sigma))}
\exp\left\{-4\pi\vol(X)\int_0^sP_\sigma(r)\,dr\right\} Z(-s,w\sigma)
\end{equation}
\cite[Theorem 3.18]{BO}.  Our formula differs from the formula in
\cite[Theorem 3.18]{BO}. This is due to the the different definition of
$P_\sigma(z)$. Since the functional equation
plays an important role in this paper, we will give a separate proof for the
functional equation of the symmetrized Selberg zeta function  
in section \ref{sect-funcequ}.

A related dynamical zeta function is the twisted Ruelle zeta function 
$R(s,\sigma)$ which is defined by
\begin{equation}\label{ruellezeta}
R(s,\sigma)=\prod_{\substack{[\gamma]\not=e\\\pr}}\left(1-\sigma(m_\gamma)
e^{-s\ell(\gamma)}\right),
\end{equation}
where, as above, $[\gamma]$ runs over the non-trivial primitive conjugacy 
classes in $\Gamma$. Note that $R(s,\sigma_0)$ equals the usual Ruelle zeta 
function
\begin{equation}\label{ruellezeta1}
R(s)=\prod_{\substack{[\gamma]\not=e\\\pr}}\left(1-e^{-s\ell(\gamma)}\right).
\end{equation}
Ruelle zeta functions of this type have been studied by Fried, and Bunke
and Olbrich \cite{BO}.  For any $\sigma\in\hat M$ let $\eta(D_X(\sigma))$ be 
the eta invariant occurring in the functional equation \eqref{functequ1}.
The two zeta functions are closely related. Namely the Ruelle zeta function 
can be expressed in terms Selberg zeta functions as follows.
\begin{lem}\label{relru-sel} 
For every $\sigma\in\hat M$ we have
\begin{equation}\label{ruselberg}
R(s,\sigma)=\frac{Z(s+1,\sigma)Z(s-1,\sigma)}
{Z(s,\sigma\otimes\sigma_{2})Z(s,\sigma\otimes\sigma_{-2})}.
\end{equation}
\end{lem}
\begin{proof}
By \cite[(3.4)]{BO} we have
\begin{equation}\label{logruelle1}
\log R(s,\sigma)=-\sum_{[\gamma]\neq e}\frac{\sigma(m_\gamma)}{n_\Gamma(\gamma)}
e^{-s\ell(\gamma)}.
\end{equation}
Using \eqref{logzeta} we get
\begin{equation}
\begin{split}
\log &Z(s+1,\sigma)+\log Z(s-1,\sigma)-\log Z(s,\sigma\otimes\sigma_2)
-\log Z(s,\sigma\otimes\sigma_{-2})\\
&=-\sum_{[\gamma]\neq e}\frac{\sigma(m_\gamma)(1-\sigma_2(m_\gamma)e^{-\ell(\gamma)}
-\sigma_{-2}(m_\gamma)e^{-\ell(\gamma)}+e^{-2\ell(\gamma)})}
{\det(\Id-\Ad(m_\gamma a_\gamma)_{\overline{\nf}})n_\Gamma(\gamma)}
e^{-s\ell(\gamma)}\\
&=-\sum_{[\gamma]\neq e}\frac{\sigma(m_\gamma)}{n_\Gamma(\gamma)}
e^{-s\ell(\gamma)}.
\end{split}
\end{equation}
Together with \eqref{logruelle1} the lemma follows.
\end{proof}
Put
\[
\theta_X(\sigma):=2\eta(D_X(\sigma))-\eta(D_X(\sigma\otimes\sigma_2))-
\eta(D_X(\sigma\otimes\sigma_{-2})),\quad \sigma\in\hat M.
\]
We summarize the main properties of $R(s,\sigma)$ by the following proposition.
\begin{prop}\label{prop-ruelle}
For each $\sigma\in\hat M$ we have
\begin{enumerate}
\item[1)]
 The infinite product (\ref{ruellezeta}) is absolutely convergent in the 
half-plane $\Re(s)>2$.
\item[2)] $R(s,\sigma)$ admits a meromorphic extension to whole complex plane.  
\item[3)] $R(s,\sigma)$ satisfies the following functional
equation.
\begin{equation}\label{functequ}
R(s,\sigma)=e^{i\pi\theta_X(\sigma)}\exp\left(4\pi^{-1}
\vol(\Gamma\bs \bH^3)s\right)R(-s,w\sigma).
\end{equation}
\end{enumerate}
\end{prop} 
\begin{proof}
1) follows from the estimation \eqref{growth}.
 The meromorphic extension is established in 
\cite[Chap. 4]{BO} and the functional equation is proved in 
\cite[Theorem 4.5]{BO}. It follows from Lemma \ref{relru-sel} and the 
functional equation  of the Selberg zeta function. Namely
using  \eqref{ruselberg} and  \eqref{functequ1} we get
\[
\begin{split}
\frac{R(s,\sigma)}{R(-s,w\sigma)}=e^{i\pi\theta_X(\sigma)}
\exp\Biggl(-4\pi\vol(X)\Biggl\{\int_0^{s+1}& P_{\sigma_k}(r)\;dr
+\int_0^{s-1} P_{\sigma_k}(r)\\
&-\int_0^{s} P_{\sigma_{k+2}}(r)\;dr -\int_0^{s} P_{\sigma_{k-2}}(r)\;dr 
\Biggr\}\Biggr).
\end{split}
\]
It follows from  (\ref{planch2}) by a simple computation that
\[
\int_0^{s+1} P_{\sigma_k}(r)\;dr
+\int_0^{s-1} P_{\sigma_k}(r)
q-\int_0^{s} P_{\sigma_{k+2}}(r)\;dr -\int_0^{s} P_{\sigma_{k-2}}(r)\;dr =
-\frac{s}{\pi^2}
\]
which implies 3).
\end{proof}

Now let $\tau\colon G\to \GL(V)$ be a representation in a finite-dimensional
complex vector space $V$. We fix a norm $\parallel\cdot\parallel$ in $V$.
The restriction $\tau|_{MA}$ of $\tau$ to $MA$ decomposes into characters:
\begin{equation}\label{decompo}
\tau|_{MA}=\bigoplus_{k\in I}\sigma_k\otimes e^{\nu_k\alpha},
\end{equation}
where $I\subset \Z$ is finite and $\nu_k\in\frac{1}{2}\Z$. Let $c=\max
\{|\nu_k|\colon k\in I\}$. Given $g\in G$, we denote by $a(g)\in A^+$
the $A^+$-component of $g$ with respect to the Cartan decomposition 
$G=KA^+K$. It follows from (\ref{decompo}) that there exists $C_1>0$ such that
\[
\parallel\tau(g)\parallel\le C_1 e^{c\alpha(\log a(g))},\quad g\in G.
\]
This implies that there exists $c_2>0$ such that
\[
\parallel \tau(\gamma)\parallel \le C e^{c_2\ell(\gamma)},\quad \gamma\in
\Gamma\setminus \{1\}.
\]
Therefore, the infinite product
\begin{equation}\label{ruelle1}
R_\tau(s)=\prod_{\substack{[\gamma]\not=e\\\pr}}\det\left(\I-\tau(\gamma)
e^{-s\ell(\gamma)}\right)
\end{equation}
is absolutely convergent in the half-plane $\Re(s)>c_2+2$. By (\ref{decompo})
we have
\[
\det\left(\I-\tau(\gamma)e^{-s\ell(\gamma)}\right)=
\det\left(\I-\tau(m_\gamma a_\gamma)e^{-s\ell(\gamma)}\right)
=\prod_{k\in I}\det\left(1-\sigma_k(m_\gamma)e^{-(s-\nu_k)\ell(\gamma)}\right).
\]
Taking the product of both sides over all non-trivial 
primitive conjugacy classes, we get
\begin{equation}\label{decompo1}
R_\tau(s)=\prod_{k\in I}R(s-\nu_k,\sigma_k),\quad \Re(s)>c_2+2.
\end{equation}
The right hand side is a meromorphic function on $\C$. This implies that
$R_\tau(s)$ admits a meromorphic continuation to $\C$. 

Using \eqref{ruselberg}, it follows that $R_\tau(s)$ can also be expressed in
terms of twisted Selberg zeta functions. This formula can be simplified using
Kostant's Bott-Borel-Weil theorem \cite{Ko} which we recall next. Let
\begin{equation}\label{exterior}
\mu_p\colon MA\to\GL(\Lambda^p\nf_\C),\quad p=0,1,2, 
\end{equation}
be the $p$-th exterior power of the adjoint representation of $MA$ on $\nf_\C$.
It decomposes into characters as follows
\begin{equation}\label{exterior1}
\mu_0=\sigma_0,\quad \mu_1=(\sigma_2\otimes e^{\alpha})\oplus 
(\sigma_{-2}\otimes e^{\alpha}),\quad \mu_2=\sigma_0\otimes e^{2\alpha}.
\end{equation}
Denote by $\tilde\mu_p$ the contragredient representation of the representation 
\eqref{exterior}. Given
$(m,n)\in\Z\times Z$, we define a character $\chi_{(m,n)}\colon MA\to \C^\times$
by
\begin{equation}\label{charact}
\chi_{(m,n)}=\sigma_{m-n}\otimes e^{\frac{m+n}{2}\alpha}.
\end{equation} 
\begin{lem}\label{thmkostant}
Let $\tau$ be an irreducible representation of $G$ with highest weight
$\Lambda_\tau\in\N_0\times \N_0$. We have the following
identity of characters of $MA$.
\begin{equation}\label{kostant1}
\sum_{p=0}^2 (-1)^p\tr\tilde\mu_p\cdot \tr\tau=\sum_{w\in W_G}(-1)^{\ell(w)}
\chi_{w(\Lambda_\tau+\rho_G)-\rho_G},
\end{equation}
where $\ell(w)$ denotes the length of $w$. 
\end{lem}
\begin{proof} 
Let $L=MA$. For a finite-dimensional $L$-module $W$ denote by $\ch_L(W)$ the
element in the character ring $R(L)$. Let $V_\tau$ be an irreducible 
$G$-module with highest weight $\Lambda_\tau$. By the analog of
a result of Kostant \cite[Theorem~5.14]{Ko}, for real Lie algebras
\cite[Theorem III.3.1]{BW}, \cite{Si}, we have
\begin{equation}\label{cohomol1}
\sum_{p=0}^2 (-1)^p\ch_L(H^p(\nf,V_\tau))=\sum_{w\in W_G}(-1)^{\ell(w)}
\chi_{w(\Lambda_\tau+\rho_G)-\rho_G},
\end{equation}
where $H^p(\nf,V_\tau)$ denotes the Lie algebra cohomology. 
By the Poincar\'e principle \cite[(7.2.3)]{Ko} we have
\begin{equation}\label{cohomol2}
\sum_{p=0}^2(-1)^p\ch_L(\Lambda^p\nf^*\otimes V_\tau)=
\sum_{p=0}^2 (-1)^p\ch_L(H^p(\nf,V_\tau)).
\end{equation}
Here $L$ acts on $\nf^*$ via the contragredient representation of the adjoint
representation. Combining \eqref{cohomol1} and \eqref{cohomol2}, the lemma
follows. In fact, in the present case the lemma could also be proved by an
elementary computation, using the parametrization \eqref{sl2irred}. 
\end{proof}
We are now ready to prove the formula which expresses $R_\tau(s)$ as a fraction
of twisted Selberg zeta functions.
For $w\in W_G$ write
\begin{equation}\label{wcharacter}
\chi_{w(\Lambda_\tau+\rho_G)-\rho_G}=\sigma_{\tau,w}\otimes 
e^{(\lambda_{\tau,w}-1)\alpha},
\end{equation}
where $\sigma_{\tau,w}\in\hat M$ and $\lambda_{\tau,w}\in\R$.

\begin{prop}
Let $\tau$ be an irreducible finite-dimensional representation of $G$.
Then we have
\begin{equation}\label{ruesel}
R_\tau(s)=\prod_{w\in W_G}Z(s-\lambda_{\tau.w},\sigma_{\tau,w})^{(-1)^{\ell(w)+1}}.
\end{equation}
\end{prop}
\begin{proof}
Recall that for an endomorphism $W$ of a finite-dimensional vector space we 
have
\begin{equation}\label{endomorph}
\det(\Id-W)=\sum_{k=0}^\infty (-1)^k\tr(\Lambda^k W).
\end{equation}
Let $m\in M$ and $a\in A$. Note that $\tilde\mu_p(ma)=
\Lambda^p\Ad(ma)_{\overline\nf}$. Hence if we apply 
\eqref{endomorph} to $\tilde\mu_p$ we get 
\[
\sum_{p=0}^2(-1)^p\tr\tilde\mu_p(ma)=\det\left(\Id-\Ad(ma)_{\overline\nf}\right).
\]
Using \eqref{kostant1} and \eqref{wcharacter}, we get
\begin{equation}\label{trace4}
\tr\tau(ma)=\sum_{w\in W_G}(-1)^{\ell(w)}\frac{\sigma_{\tau,w}(m)}
{\det\left(\Id-\Ad(ma)_{\overline\nf}\right)}e^{(\lambda_{\tau,w}-1)\alpha(\log a)}.
\end{equation}
Next we have
\begin{equation}\label{logruelle}
\begin{split}
\log R_\tau(s)&=\sum_{\substack{[\gamma]\neq e\\\OP{prime}}}
\tr \log\left(\Id-\tau(\gamma)e^{-s\ell(\gamma)}\right)\\
&=-\sum_{\substack{[\gamma]\neq e\\\OP{prime}}}\sum_{k=1}^\infty
\frac{\tr\left(\tau(\gamma)e^{-s\ell(\gamma)}\right)^k}{k}\\
&=-\sum_{[\gamma]\not= e}\frac{\tr\tau(\gamma)}{n_\Gamma(\gamma)}
e^{-s\ell(\gamma)}.
\end{split}
\end{equation}
Now let $\gamma\in\Gamma\setminus\{e\}$, $\gamma\sim m_\gamma a_\gamma$. Then
$\log a_\gamma=\ell(\gamma)H$. 
Inserting \eqref{trace4} on the right hand side of \eqref{logruelle}, we get
\[
\log R_\tau(s)=\sum_{w\in W_G}(-1)^{\ell(w)+1}\sum_{[\gamma]\neq e}
\frac{\sigma_{\tau,w}(m_\gamma)}
{\det\left(\Id-\Ad(m_\gamma a_\gamma)_{\overline\nf}\right)n_\gamma(\gamma)}
e^{-(s-\lambda_{\tau,w}+1)\ell(\gamma)}
\]
By \cite[(3.6)]{BO}, the right hand side equals
\[
\sum_{w\in W_G}(-1)^{\ell(w)+1}\log Z(s-\lambda_{\tau,w},\sigma_{\tau,w}),
\]
which proves the proposition.
\end{proof}
We also need to consider symmetrized Ruelle and Selberg zeta functions. Recall
that the nontrivial element $w_A\in W_A$ acts on $\hat M$ by $w_A\sigma_k=
\sigma_{-k}$. Let $\sigma\in\hat M\setminus\{\sigma_0\}$. Put
\begin{equation}\label{symmsel}
S(s,\sigma):=Z(s,\sigma)Z(s,w_A\sigma).
\end{equation}
This is the symmetrized Selberg zeta function. Let $\theta\colon G\to G$ be 
the Cartan involution. Put
\begin{equation}\label{thetatau}
\tau_\theta=\tau\circ\theta.
\end{equation}
Note that  $\tau_p=\Sym^p$ satisfies
$\tau_p\circ\theta=\overline\tau_p,\quad\mathrm{and}\quad\overline\tau_p\circ
\theta=\tau_p$. Thus $\theta$ acts on the highest weights by 
\begin{equation}\label{highestw}
\theta(m,n)=(n,m), \quad (m,n)\in\N_0\times\N_0.
\end{equation} 
By \eqref{sl2irred} it follows that an irreducible finite-dimensional
representation $\tau$ of $G$ satisfies $\tau_\theta=\tau$, if and only if 
$\tau=\tau_{m,m}$ for some $m\in\N_0$.
\begin{prop}
Let $\tau$ be an irreducible finite-dimensional representation of $G$.
Then we have
\begin{equation}\label{ruellesymm}
R_\tau(s)R_{\tau_\theta}(s)=\prod_{w\in W_G} S(s-\lambda_{\tau,w},\sigma_{\tau,w})
^{(-1)^{\ell(w)+1}},\quad \tau_\theta\ncong\tau,
\end{equation}
and
\begin{equation}\label{ruelle5}
R_{\tau_{m,m}}(s)=Z(s-(m+1),\sigma_0)Z(s+m+1,\sigma_0)S(s,\sigma_{2m+2})^{-1},
\quad m\in\N_0.
\end{equation}
\end{prop}
\begin{proof}
Put
\[
\Xi(\tau)=\{w(\Lambda_\tau+\rho_G)-\rho_G\colon w\in W_G\}.
\]
Let $\tau=\tau_{m,n}$. Then we have
\[
\Xi(\tau)=\{(m,n),(-(m+2),n),(m,-(n+2)),(-(m+2),-(n+2))\}.
\]
By \eqref{charact} and \eqref{wcharacter}, it follows that
\begin{equation}\label{characterw1}
\begin{split}
\left\{(\sigma_{\tau,w},\lambda_{\tau,w})\colon w\in W_G\right\}=
\bigl\{&\left(\sigma_{m-n},(m+n)/2+1\right),
\left(\sigma_{-(m+n+2)},(n-m)/2\right),\\
&\left(\sigma_{m+n+2},(m-n)/2\right),
\left(\sigma_{n-m},-(m+n)/2-1\right)\bigr\}.
\end{split}
\end{equation}
Assume that $m\neq n$. Using that $(\tau_{m,n})_\theta=\tau_{n,m}$ and
\eqref{characterw1}, it follows  that
\[
\begin{split}
\left\{(\sigma_{\tau,w},\lambda_{\tau,w})\colon w\in W_G\right\}\cup
&\left\{(\sigma_{\tau_\theta,w},\lambda_{\tau_\theta,w})\colon w\in W_G\right\}\\
&\hskip1truecm=\left\{(\sigma_{\tau,w},\lambda_{\tau,w}), 
(w_A\sigma_{\tau,w},\lambda_{\tau,w})\colon w\in W_G\right\}.
\end{split}
\]
By \eqref{ruesel} the first equality follows. 
 Now assume that $\tau_\theta=\tau$ . By \eqref{highestw} there exists
$m\in\N_0$ such that $\tau=\tau_{m,m}$. In this case we get 
\begin{equation}\label{characterw2}
\left\{(\sigma_{\tau,w},\lambda_{\tau,w})\colon w\in W_G\right\}=\left\{
(\sigma_0,m+1),(\sigma_{-2(m+1)},0),(\sigma_{2(m+1)},0),(\sigma_0,-(m+1)
\right\}.
\end{equation}
Using again \eqref{ruesel} and \eqref{symmsel}, we get \eqref{ruelle5}.
\end{proof}

\section{Bochner-Laplace  operators}\label{sec-bochlapl}
\setcounter{equation}{0}

In this section we study certain auxiliary elliptic operators
which are needed to derive the determinant formula and the functional
equation for the Selberg zeta function. These operators were first introduced
by Bunke and Olbrich \cite{BO}.

Let $w_A\in W_A$ be the nontrivial
element. It acts on $\sigma_k\in\hat M$ by $w_A\sigma_k=\sigma_{-k}$. Thus, if
$k\neq 0$, then $\sigma_k$ is not $W_A$-invariant. For $l\in\N_0$ let $\nu_l\in
\hat K$ denote the irreducible representation of $K=\SU(2)$ of highest
weight $l$. Then we have
\begin{equation}\label{restrict1}
\nu_l|_M=\bigoplus_{k=0}^l\sigma_{l-2k}.
\end{equation}
Let $R(K)$ and $R(M)$ denote the representation rings of $K$ and $M$, 
respectively. The inclusion $i\colon M\to K$ induces the restriction map
$i^*\colon R(K)\to R(M)$. From  \eqref{restrict1} we get
\begin{equation}\label{restrict2}
\begin{aligned}
&i^*(\nu_l-\nu_{l-2})=\sigma_l +\sigma_{-l},\quad l\in\N,\; l\ge 2;\\
&i^*(\nu_1)=\sigma_1 +\sigma_{-1},\quad i^*(\nu_0)=\sigma_0.
\end{aligned}
\end{equation}
It follows from \eqref{restrict2} that for every $\sigma\in\hat M$ there 
exists a unique $\xi_\sigma\in R(K)$ such that 
\begin{equation}\label{restrict3}
i^*(\xi_\sigma)=\sigma+w_A\sigma.
\end{equation}
Then  we have
\begin{equation}\label{multipl}
\xi_\sigma=\sum_{\nu\in\hat K} m_\nu(\sigma) \nu.
\end{equation}
with $m_\nu(\sigma)\in\{0,\pm1\}$ for $\sigma\neq \sigma_0$ and 
$m_{\nu_{l}}(\sigma_0)=0$, if $l\neq0$, and $m_{\nu_{0}}(\sigma_0)=2$.

Given  $\nu\in\hat K$, let $\tilde E_\nu$ denote the associated homogeneous
vector bundle over $G/K$ and $E_\nu=\Gamma\bs \tilde E_\nu$ the 
corresponding locally homogeneous bundle over $X$. 
For $\sigma\in\hat M$ and $\nu\in\hat K$ let $m_\nu(\sigma)$ be defined by
\eqref{multipl}.  Put
\begin{equation}\label{bundle}
E(\sigma)=\bigoplus_{\substack{\nu\\m_\nu(\sigma)\neq 0}}E_\nu.
\end{equation}
This bundle has a canonical grading 
\begin{equation}\label{gradedb}
E(\sigma)=E^+(\sigma)\oplus E^-(\sigma)
\end{equation}
defined by the sign of $m_\nu(\sigma)$.

Let $\tilde A_\nu$ be the elliptic $G$-invariant differential operator on
 $C^\infty(G/K,\tilde E_\nu)\cong(C^\infty(G)\otimes V_\nu)^K$ which is induced 
by $-\Omega$, where $\Omega\in\cZ(\gf_\C)$ is the Casimir element. Let 
\[
\tilde\Delta_\nu=(\nabla^\nu)^*\nabla^\nu
\]
be the connection
Laplacian associated to the canonical invariant connection $\nabla^\nu$
of $\tilde E_\nu$. By \cite[Proposition 1.1]{Mia} we have
\begin{equation}\label{bochlapl1}
\tilde A_\nu=\tilde\Delta_\nu-\nu(\Omega_K),
\end{equation}
where $\Omega_K\in\cZ(\kf_\C)$ is the Casimir element of $K$. Being 
$G$-invariant, $\tilde A_\nu$ descends to an elliptic operator 
\begin{equation}\label{aoperator1}
A_\nu\colon C^\infty(X,E_\nu)\to C^\infty(X,E_\nu).
\end{equation}
It follows from \eqref{bochlapl1} that $A_\nu$
is symmetric and bounded from below. For $l\in\Z$ put
\begin{equation}\label{cl}
c(\sigma_l)=\frac{l^2}{4}-1.
\end{equation}
Define the operator $A(\sigma)$ acting on 
$C^\infty(X,E(\sigma))$ by
\begin{equation}\label{aoperator}
A(\sigma):=\bigoplus_{\substack{\nu\\m_\nu(\sigma)\neq 0}}A_\nu+c(\sigma).
\end{equation}
Obviously, $A(\sigma)$  preserves the grading  of $E(\sigma)$.

By \eqref{bochlapl1} the elliptic operator $A_\nu$ is symmetric and bounded 
from below. Therefore the
heat operator $e^{-tA_\nu}$ is well defined and is a trace class operator. 
Given $\sigma\in\hat M$, put
\begin{equation}\label{virtual1}
K(t;\sigma)=\sum_{\nu\in\hat K} m_\nu(\sigma)
\Tr(e^{-tA_\nu}),
\end{equation}
where $m_\nu(\sigma)$ is defined by \eqref{multipl}.
Our next goal is to  use the Selberg trace formula to express
$K(t,\sigma)$ in terms of the length of the closed geodesics.

Let $\tilde A_\nu$ be the lift of $A_\nu$ to the universal covering 
$\tilde X=G/K$. It acts in the
space of smooth sections of the homogeneous vector bundle $\tilde E_\nu$ 
associated to $\nu$. With respect to the isomorphism $C^\infty(G/K,\tilde E_\nu)
\cong (C^\infty(G)\otimes V_\nu)^K$ we have
\[
\tilde A_\nu=-R(\Omega)\otimes \Id_{V_\nu}.
\]
Let $e^{-t\tilde A_\nu}$, $t>0$, the heat semigroup generated by $\tilde A_\nu$.
This is a smoothing operator on $L^2(G/K,\tilde E_\nu)\cong 
(L^2(G)\otimes V_\nu)^K$ which commutes with the action of $G$. Therefore it
is of the form
\[
\left(e^{-t\tilde A_\nu}\phi\right)(g)=\int_G H_t^\nu(g^{-1}g^\prime)\phi(g^\prime)
\;dg^\prime,\quad \phi\in (L^2(G)\otimes V_\nu)^K,\quad g\in G,
\]
where the kernel $H_t^\nu\colon G\to \End(V_\nu)$ is $C^\infty$, $L^2$, and 
satisfies the covariance property
\begin{equation}\label{covar2}
H_t^\nu(k^{-1}gk^\prime)=\nu(k)^{-1}\circ H_t^\nu(g)\circ \nu(k^\prime),\quad 
k,k^\prime\in K,\; g\in G.
\end{equation}
Actually, a much stronger result holds. For $q>0$ let $\Co^q(G)$ be 
Harish-Cahndra's $L^p$-Schwartz space. Then we have
\begin{equation}\label{schwartz}
H_t^\nu \in (\mathcal{C}^q(G)\otimes\End(V_\tau))^{K\times K}
\end{equation}
for all $q>0$. The proof is similar to the proof of Proposition 2.4 in 
\cite{BM}. By standard arguments it
follows that the kernel of the heat operator $e^{-tA_\nu}$ is given by
\begin{equation}\label{eq:15}
  H^{\nu}(t;x,x')=\sum_{\gamma\in\Gamma}H^{\nu}_t(g^{-1}\gamma g'),
\end{equation}
where $x,x'\in X$ and  $x=\Gamma g K$ and $x'=\Gamma g'K$. Therefore the
trace of the heat operator $e^{-tA_\nu}$ is given by
\begin{displaymath}
\Tr\left(e^{-tA_\nu}\right)=\int_X\tr H^{\nu}(t;x,x)\; dx,
  \end{displaymath}
where $\tr$ denotes the trace $\tr\colon\OP{End}(
E_{\nu,x})\to \C$ for $x\in X$. Let 
\[
h^\nu_t(g)=\tr H^{\nu}_t(g).
\] 
Using \eqref{covar2} and \eqref{eq:15}, it follows that
\begin{equation}\label{trace1}
\Tr\left(e^{-tA_\nu}\right)=\int_{\Gamma\bs G}\sum_{\gamma\in \Gamma} 
h^{\nu}_t(g^{-1}\gamma g)\,dg.
\end{equation}
Let $R_\Gamma$ denote the right regular representation of $G$ on
$L^2(\Gamma\bs G)$. Then \eqref{trace1} can be written as
\begin{equation}\label{trace2}
  \Tr\left(e^{-tA_\nu}\right)=\Tr R_\Gamma(h^{\nu}_t).
\end{equation}
Let
\begin{equation}\label{virtual3}
h^\sigma_t=\sum_{\nu\in \hat K}m_\nu(\sigma) h^\nu_t.
\end{equation}
Then by \eqref{virtual1} and  \eqref{trace2} we get
\[
K(t;\sigma)=\Tr R_\Gamma\left(h^\sigma_t\right),\quad t>0.
\]
We can now apply the Selberg trace formula \cite{Wa1} . We use the notation
introduced in section \ref{prelim}. Let $\overline\nf=\theta(\nf)$, where 
$\theta$ is the Cartan involution.  For $\gamma\in\Gamma\setminus\{e\}$ put
\[
 D(\gamma)=e^{\ell(\gamma)}
\det(\Id-\Ad(m_\gamma a_\gamma)_{\overline{\nf}}). 
\]
Then the Selberg trace formula gives
\begin{equation}\label{anator3}
  \begin{split}    
K(t;\sigma)  &=\OP{Vol}(X)h^\sigma_t(e)\\
  &\mspace{30mu} +\frac{1}{2\pi}\sum_{[\gamma]\neq e}
\frac{\ell(\gamma)}{n_\Gamma(\gamma)D(\gamma)}\sum_{n\in\Z}
  \overline{\sigma_n(m_\gamma)}\int_\R\Theta_{n,\lambda}(h^\sigma_t)
e^{-i\ell(\gamma)\lambda}\;d\lambda.
  \end{split}
\end{equation}
Note that $h_t^\sigma(e)$ can also be expressed in terms of characters. By
\eqref{schwartz},  each $h^\nu_t$ belongs to $\mathcal{C}^q(G)$ for
all $q>0$. Therefore $h_t^\sigma$ is in $\mathcal{C}^q(G)$. Hence we can  
apply the Plancherel  formula for
$G$ (see \cite[Theorem 11.2]{Kn}).  With respect to the normalizations of Haar
measures used in \cite{Kn} and the definition of the Plancherel polynom 
\eqref{planch2}, we have
\begin{equation}\label{plachform}
  h^\sigma_t(e)=\sum_{n\in\Z}\int_{\mathbb R}
\Theta_{n,\lambda}(h^\sigma_t)P_{\sigma_n}(i\lambda)\,d\lambda.
\end{equation}

To continue we need to compute the characters $\Theta_{n,\lambda}(h^\sigma_t)$.
First by \eqref{virtual3} we have
\begin{equation}\label{virtual4}
\Theta_{n,\lambda}(h^\sigma_t)=\sum_{\nu\in \hat K}m_\nu(\sigma)
\Theta_{n,\lambda}(h^\nu_t),
\end{equation}
which reduces the problem to the computation of $\Theta_{n,\lambda}(h^\nu_t)$.
For any unitary representation $\pi$  of $G$ on a Hilbert
space $\H_\pi$ set
\[
\tilde \pi(H^{\nu}_t)=\int_G \pi(g)\otimes H^{\nu}_t(g)\,dg.
\]
This defines a bounded operator on 
$\H_\pi\otimes V_\nu$. As in \cite[pp. 160-161]{BM} it follows from 
\eqref{covar2} that relative to the splitting
\[
\H_\pi\otimes  V_\nu=
\left(\H_\pi\otimes  V_\nu\right)^K\oplus
\left[\left(\H_\pi\otimes  V_\nu\right)^K
\right]^\perp,
\]
$\tilde \pi(H^{\nu}_t)$ has the form
\[
\tilde \pi(H^{\nu}_t)=\begin{pmatrix}\pi(H^{\nu}_t)& 0\\ 0& 0
\end{pmatrix}
\]
with $\pi(H^{\nu}_t)$ acting on $\left(\H_\pi\otimes  V_\tau\right)^K$. 
Then it follows as in \cite[Corollary 2.2]{BM} that
\begin{equation}\label{integop}
\pi(H^{\nu}_t)=e^{t\pi(\Omega)}\Id,
\end{equation}
where $\Id$ is the identity on $\left(\H_\pi\otimes  V_\nu\right)^K$.
Let $\{\xi_n\}_{n\in\N}$ and $\{e_j\}_{j=1}^m$ be orthonormal bases of $\H_\pi$
and $V_\nu$, respectively. Then we have
\begin{equation}\label{equtrace}
\begin{split}
\Tr\pi(H_t^\nu)&=\sum_{n=1}^\infty\sum_{j=1}^m\langle\pi(H_t^\nu)
(\xi_n\otimes e_j),(\xi_n\otimes e_j)\rangle\\
&=\sum_{n=1}^\infty\sum_{j=1}^m\int_G\langle\pi(g)\xi_n,\xi_n\rangle
\langle H_t^\nu(g)e_j,e_j\rangle\,dg\\
&=\sum_{n=1}^\infty\int_G h_t^\nu(g)\langle\pi(g)\xi_n,\xi_n\rangle\;dg\\
&=\Tr\pi(h_t^\nu).
\end{split}
\end{equation}
Together with \eqref{integop} we get
\begin{equation}\label{globalch}
\Tr\pi(h_t^\nu)=e^{t\pi(\Omega)}\dim\left(\H_\pi\otimes  
V_\nu\right)^K.
\end{equation}
Now we consider a unitary principal series representation $\pi_{n,\lambda}$
Let $[\nu|_M:\sigma_n]$ denote the multiplicity of $\sigma_n\in\hat M$ in 
$\nu|_M$. It equals $0$ or $1$. For any representation $\pi$ of $G$ denote by
$\pi^\vee$ the contragredient representation of $\pi$. 
By Frobenius reciprocity \cite[p. 208]{Kn}, we 
have
\[
\dim\left(\H_{n,\lambda}\otimes V_\nu\right)^K=[\pi_{n,\lambda}^\vee|_K\colon \nu]=
[\pi_{-n,-\lambda}|_K\colon \nu]=[\nu|_M\colon \sigma_{-n}]=
[\nu|_M\colon \sigma_n].
\]
Combined with \eqref{globalch}, we obtain
\[
\Theta_{n,\lambda}(h^\nu_t)=e^{t\pi_{n,\lambda}(\Omega)}
[\nu|_M:\sigma_n].
\]
Using \eqref{virtual4}, \eqref{multipl} and \eqref{restrict3}, we get
\begin{equation}
\begin{split}\label{globalch1}
\Theta_{n,\lambda}(h^\sigma_t)=e^{t\pi_{n,\lambda}(\Omega)}
\sum_{\nu\in\hat K} m_\nu(\sigma)[\nu|_M:\sigma_n]
=e^{t\pi_{n,\lambda}(\Omega)}[\sigma+w_A\sigma:\sigma_n].
\end{split}
\end{equation}
The Casimir eigenvalue  $\pi_{n,\lambda}(\Omega)$ is given by \eqref{ucasimir}.
Using the definition of $c(\sigma)$ by \eqref{cl} it can be written as 
\begin{equation}\label{ucasimir1}
\pi_{n,\lambda}(\Omega)=-\lambda^2+c(\sigma_n).
\end{equation}
Now we can put our computations together. Let $k\in\N_0$. 
If we insert \eqref{globalch1} in \eqref{anator3} and \eqref{plachform} and
use \eqref{ucasimir1}, we get
\begin{equation}\label{virtual11}
\begin{split}
K(t;\sigma_k)=e^{tc(\sigma_k)}\Biggl(2\vol(X)
&\int_\R e^{-t\lambda^2}P_{\sigma_k}(i\lambda)\;d\lambda\\
&+\sum_{[\gamma]\neq e}\frac{\ell(\gamma)}{n_\Gamma(\gamma)}
L_{\sym}(\gamma;\sigma_k)\frac{e^{-\ell(\gamma)^2/(4t)}}{(4\pi t)^{1/2}}\Biggr),
\end{split}
\end{equation}
where 
\begin{equation}\label{lefschetz}
L_{\sym}(\gamma,\sigma)=\frac{(\sigma(m_\gamma)+(w_A\sigma)(m_\gamma))
e^{-\ell(\gamma)}}
{\det\left(\Id-\Ad(m_\gamma a_\gamma)_{\overline \nf}\right)}.
\end{equation}
Using the definition of $A(\sigma)$ by \eqref{aoperator} together with
\eqref{virtual1}, we finally get
\begin{prop}\label{prop-virtual}
For every $\sigma\in\hat M$ we have
\begin{equation}\label{virtual8}
\Tr_s\left(e^{-tA(\sigma)}\right)=2\vol(X)
\int_\R e^{-t\lambda^2}P_{\sigma}(i\lambda)\;d\lambda
+\sum_{[\gamma]\neq e}\frac{\ell(\gamma)}{n_\Gamma(\gamma)}
L_{\sym}(\gamma;\sigma)\frac{e^{-\ell(\gamma)^2/(4t)}}{(4\pi t)^{1/2}}.
\end{equation}
\end{prop}

\section{The functional equation of the Selberg zeta function}
\label{sect-funcequ}
\setcounter{equation}{0}

One of the main ingredients of the proof of Theorem \ref{asymptotic} is the
functional equation \eqref{functequ1} satisfied by the Selberg zeta
function $Z(s,\sigma)$. In particular, it is important to determine the sign 
in the exponential factor.  We include a proof of the functional 
equation for the symmetrized Selberg zeta function which suffices for our
purpose.

Let $\sigma\in\hat M$. Note that $A(\sigma)$ is a second order elliptic 
differential operator on a compact manifold. Therefore it is essentially
self-adjoint and the unique self-adjoint extension of $A(\sigma)$
has pure point spectrum consisting of a sequence of eigenvalues 
$\lambda_1\le\lambda_2\le\cdots\to\infty$ of finite multiplicities. It follows
from Weyl's law that
\begin{equation}\label{weyllaw}
\sum_{\lambda_i>0}\lambda_i^{-2}<\infty.
\end{equation}
Therefore the resolvent $(A(\sigma)+s^2)^{-1}$, $\Re(s^2)\gg 0$, is a 
Hilbert-Schmidt operator. Let $\Re(s_i^2)\gg0$ ,$i=1,2$. By the resolvent 
equation we have
\[
(A(\sigma)+s_1^2)^{-1}-(A(\sigma)+s_2^2)^{-1}=(s_2^2-s_1^2)
(A(\sigma)+s_1^2)^{-1}\circ (A(\sigma)+s_2^2)^{-1}.
\]
Thus the right hand side is a product of Hilbert-Schmidt operators and 
therefore, it is a trace class operator. 
Hence $(A(\sigma)+s_1^2)^{-1}-(A(\sigma)+s_2^2)^{-1}$ is a trace class 
operator. Now observe that
\[
(A(\sigma)+s^2)^{-1}=\int_0^\infty e^{-ts^2}e^{-tA(\sigma)}\;dt.
\]
Furthermore we have the heat expansion 
\begin{equation}\label{heatexp}
\Tr\left(e^{-tA(\sigma)}\right)\sim\sum_{j\ge0} a_jt^{-3/2+j}
\end{equation}
as $t\to+0$. Let $\Re(s^2),\Re(s_0^2)\gg0$. Then it follows from \eqref{heatexp} 
that
\begin{equation}\label{resolvent}
\Tr_s\left((A(\sigma)+s^2)^{-1}-(A(\sigma)+s_0^2)^{-1}\right)=
\int_0^\infty(e^{-ts^2}-e^{-ts_0^2})\Tr_s\left(e^{-tA(\sigma)}\right)\;dt.
\end{equation}
Now we replace $\Tr_s\left(e^{-tA(\sigma)}\right)$ by the right hand side of 
\eqref{virtual8}. First note that for 
$\Re(s)>0$ we have
\begin{equation}\label{lapltr1}
\int_0^\infty e^{-ts^2}\frac{e^{-\ell(\gamma)^2/(4t)}}{\sqrt{4\pi t}}=
\frac{1}{2s}e^{-s\ell(\gamma)}
\end{equation}
Furthermore by Cauchy's theorem we have
\begin{equation}\label{lapltr2}
\begin{split}
\int_0^\infty (e^{-ts^2}-e^{-ts_0^2})
\left(\int_\R e^{-t\lambda^2}P_\sigma(i\lambda)\;d\lambda\right)\;dt
&=\int_\R\frac{s_0^2-s^2}{(\lambda^2+s^2)(\lambda^2+s_0^2)}P(i\lambda)
\;d\lambda\\
&=\frac{\pi}{s}P_\sigma(s)-\frac{\pi}{s_0}P_\sigma(s_0).
\end{split}
\end{equation}
For the last equality we used that $P_\sigma(s)$ is an even polynomial.
By \eqref{lapltr1} and \eqref{lapltr2}, we get
\begin{equation}
\begin{split}
&\Tr_s\left((A(\sigma)+s^2)^{-1}-(A(\sigma)+s_0^2)^{-1}\right)=
2\pi\vol(X) \left(\frac{P_\sigma(s)}{s}-\frac{P_\sigma(s_0)}{s_0}\right)\\
&\hskip3truecm+\frac{1}{2s}\sum_{[\gamma]\neq e}\frac{\ell(\gamma)}
{n_\Gamma(\gamma)}
L_{\sym}(\gamma;\sigma)e^{-s\ell(\gamma)}-\frac{1}{2s_0}\sum_{[\gamma]\neq e}
\frac{\ell(\gamma)}{n_\Gamma(\gamma)}L_{\sym}(\gamma;\sigma)e^{-s_0\ell(\gamma)}.
\end{split}
\end{equation}
By \eqref{logzeta} we have
\[
\sum_{[\gamma]\neq e}\frac{\ell(\gamma)}{n_\Gamma(\gamma)}
L_{\sym}(\gamma;\sigma)e^{-s\ell(\gamma)}=\frac{Z^\prime(s,\sigma)}{Z(s,\sigma)}
+\frac{Z^\prime(s,w_A\sigma)}{Z(s,w_A\sigma)},
\]
which is the logarithmic derivative of the symmetrized Selberg zeta function
$S(s,\sigma)$ defined by \eqref{symmsel}. Thus we get
\begin{equation}\label{lapltr3}
\begin{split}
\Tr_s\left((A(\sigma)+s^2)^{-1}-(A(\sigma)+s_0^2)^{-1}\right)=&
2\pi\vol(X) \left(\frac{P_{\sigma}(s)}{s}-\frac{P_{\sigma}(s_0)}{s_0}\right)\\
&+\frac{1}{2s}\frac{S^\prime(s,\sigma)}{S(s,\sigma)}
-\frac{1}{2s_0}\frac{S^\prime(s_0,\sigma)}{S(s_0,\sigma)}.
\end{split}
\end{equation}
Put
\[
\Xi(s,\sigma)=\exp\left(4\pi\vol(X)\int_0^s P_\sigma(r)\;dr\right)\;
S(s,\sigma).
\]
Then \eqref{lapltr3} can be rewritten as
\begin{equation}\label{lapltr7}
\Tr_s\left((A(\sigma)+s^2)^{-1}-(A(\sigma)+s_0^2)^{-1}\right)=
\frac{1}{2s}\frac{\Xi^\prime(s,\sigma)}{\Xi(s,\sigma)}-
\frac{1}{2s_0}\frac{\Xi^\prime(s_0,\sigma)}{\Xi(s_0,\sigma)}.
\end{equation}
From this equality one can deduce the existence of the meromorphic extension
of $S(s,\sigma)$ and determine the location of the singularities, i.e.,
zeros and poles of $S(s,\sigma)$. Let $\lambda_1<\lambda_2<\cdots$ be the
eigenvalues of $A(\sigma)$. For each $\lambda_j$ let $\E(\lambda_j)$ be
the eigenspace of $A(\sigma)$ with eigenvalue $\lambda_j$. Put
\[
m_s(\lambda_j,\sigma)=\dim_{\gr}\E(\lambda_j).
\]
If $\lambda_j<0$, we choose the square root $\sqrt{\lambda_j}$ which has
positive imaginary part. Put
\[
s_j^\pm=\pm i\sqrt{\lambda_j},\quad j\in\N.
\]
\begin{prop}\label{prop-merom}
The Selberg zeta function $S(s,\sigma)$, defined for $\Re(s)>2$ by 
\eqref{selzeta}, has a meromorphic extension to $s\in\C$. 
The set of singularities of $S(s,\sigma)$ equals $\{s_j^\pm\colon j\in\N\}$. 
If $\lambda_j\neq 0$, then  the order of $S(s,\sigma)$ at both $s_j^+$ and 
$s_j^-$ is equal to $m_s(\lambda_j,\sigma)$. The order of the singularity at
$s=0$ is $2m_s(0,\sigma)$. 
\end{prop}
\begin{proof}
The left hand side of \eqref{lapltr7} equals
\[
\sum_{j=1}^\infty m_s(\lambda_j,\sigma)\left\{\frac{1}{s^2+\lambda_j}-
\frac{1}{s_0^2+\lambda_j}\right\}.
\]
By \eqref{weyllaw} the series converges absolutely and uniformly on compact
subsets which shows that it is a meromorphic function of $s\in\C$ and
the only poles are simple and occur exactly at the points $\{s_j^\pm\colon
j\in\N\}$. Hence the logarithmic derivative of $\Xi(s,\sigma)$ is a meromorphic
function with the same poles. Let $\lambda_j\neq0$. Then
\[
\frac{2s}{s^2+\lambda_j}=\frac{1}{s-s_j^+}+
\frac{1}{s-s_j^-}.
\]
It follows that $s_j^\pm$ are simple poles of $\Xi^\prime(s,\sigma)\cdot
\Xi(s,\sigma)^{-1}$ with residue $m_s(\lambda_j,\sigma)$. Hence the order of
$\Xi(s,\sigma)$ at $s^\pm_j$ equals $m_s(\lambda_j,\sigma)$. In the same way
it follows that the order of $\Xi(s,\sigma)$ at $s=0$ equals $2m_s(0,\sigma)$.
\end{proof}
Now subtract from \eqref{lapltr7} the same equation for $-s$ and multiply by 
$2s$. Then we get
\[
\frac{\Xi^\prime(s,\sigma)}{\Xi(s,\sigma)}+\frac{\Xi^\prime(-s,\sigma)}
{\Xi(-s,\sigma)}=0,
\]
which shows that the logarithmic derivative of $\Xi(s)\cdot\Xi(-s)^{-1}$ 
equals zero. 
Therefore $\Xi(s)\cdot\Xi(-s)^{-1}$ is 
constant. By Proposition \ref{prop-merom}  the order of $S(s,\sigma)$ at zero 
is even. Hence
\[
\lim_{s\to 0}\frac{\Xi(s)}{\Xi(-s)}=1.
\]
This implies $\Xi(s)=\Xi(-s)$. Since $P_\sigma(z)$ is even, we obtain the 
following functional equation for $S(s,\sigma)$:
\begin{equation}\label{functequ3}
S(s,\sigma)=\exp\left(-8\pi\vol(X)\int_0^s P_\sigma(r)\;dr\right)\;
S(-s,\sigma).
\end{equation}
Note that $\overline{Z(s,\sigma_m)}=Z(\overline s,\sigma_{-m})$. Hence for 
$s\in\R$ we have $S(s,\sigma)=|Z(s,\sigma)|^2$. Then \eqref{functequ3} is
reduced to
\begin{equation}\label{functequ4}
|Z(s,\sigma)|=\exp\left(-4\pi\vol(X)\int_0^s P_\sigma(r)\;dr\right)\;
|Z(-s,\sigma)|,\quad s\in\R.
\end{equation}

\section{The determinant formula}\label{detformula}
\setcounter{equation}{0}

By \cite[Theorem 3.19]{BO} the twisted Selberg zeta 
function can be expressed as a graded regularized determinant of $A(\sigma)$.
We include a simple proof of this formula for our case.

First we recall the notion of the graded regularized determinant of an 
elliptic self-adjoint operator.
Let $E=E^+\oplus E^-$ be a $\Z/2\Z$-graded Hermitian vector bundle over a 
compact
Riemannian manifold. Let $P\colon C^\infty(Y,E)\to  C^\infty(Y,E)$ be an 
elliptic differential operator which is symmetric and bounded from below. 
Assume that $P$ 
preserves the grading, i.e., assume that with respect to the decomposition
\[
C^\infty(Y,E)=C^+(Y,E^+)\oplus C^\infty(Y,E^-)
\]
$P$ takes the form 
\[
P=\begin{pmatrix}P^+&0\\0&P^-\end{pmatrix}. 
\]
Then we define the graded determinant $\det_\gr(P)$ of $P$ by
\begin{equation}\label{grdet}
\det_{\gr}(P)=\frac{\det(P^+)}{\det(P^-)}.
\end{equation}
Given $\sigma\in\hat M$, let $A(\sigma)$ be the elliptic operator defined by
\eqref{aoperator}. It acts in a graded vector bundle. Hence the  graded 
determinant $\det_\gr(s^2+A(\sigma))$ is defined. Let 
$P_{\sigma}(r)$ be the
Plancherel polynomial \eqref{planch2}. By \cite[Theorem 3.19]{BO} the twisted
symmetrized Selberg zeta function $S(s,\sigma)$ can be expressed by the graded 
determinant as follows.
\begin{prop}
We have
\begin{equation}\label{seldet1}
S(s;\sigma)=\det_\gr\left(s^2+A(\sigma)\right)\exp\left(-4\pi\vol(X)
\int_0^sP_{\sigma}(r)\;dr\right),\quad \sigma\neq\sigma_0,
\end{equation} 
and
\begin{equation}\label{seldet2}
Z(s;\sigma_0)=\det\left(s^2-1+\Delta\right)\exp\left((6\pi)^{-1}\vol(X)
s^3\right),
\end{equation}
where $\Delta$ is the Laplace operator on $C^\infty(X)$ and $\det$ is the
usual regularized determinant. 
\end{prop} 
\begin{proof} We give a simple proof of this formula. Let
$\sigma\in\hat M$ . For $\Re(s^2)\gg 0$ let
\begin{equation}\label{dzeta}
\zeta(z,s)=\int_0^\infty e^{-ts^2}\Tr_s\left(e^{-tA(\sigma)}\right)t^{z-1}\;dt.
\end{equation}
The integral converges absolutely and uniformly on compact subsets of the
half-plane $\Re(z)>3/2$.  It admits an 
extension to a meromorphic function of $z\in\C$ which is differentiable
in $s$. It is regular at $z=0$ and we have
\begin{equation}\label{lauexp2}
\zeta(z,s)= -\log \det(A(\sigma)+s^2)+O(z).
\end{equation}
Furthermore for $\Re(z)>3/2$ we have
\begin{equation}
-\frac{1}{2s}\frac{d}{ds}\zeta(z,s)=\int_0^\infty e^{-ts^2}
\Tr_s\left(e^{-tA(\sigma)}\right)t^{z}\;dt.
\end{equation}
Thus by \eqref{lauexp2} we get 
\begin{equation}
\begin{split}
\frac{1}{2s}\frac{d}{ds}\log\det_{\gr}\left(A(\sigma)+s^2\right)&-
\frac{1}{2s_0}\frac{d}{ds}\log\det_{\gr}\left(A(\sigma)+s^2\right)\big|_{s=s_0}\\
&=\lim_{z\to0}\left(-\frac{1}{2s}\frac{d}{ds}\zeta(z,s)
+\frac{1}{2s_0}\frac{d}{ds}\zeta(z,s)\big|_{s=s_0}\right)\\
&=\int_0^\infty(e^{-ts^2}-e^{-ts_0^2})\Tr_s\left(e^{-tA(\sigma)}\right)\;dt.
\end{split}
\end{equation}
Assume that $\sigma\neq\sigma_0$. 
Together with \eqref{resolvent} and \eqref{lapltr3} we get 
\begin{equation}
\frac{d}{ds}\log\det_{\gr}\left(A(\sigma)+s^2\right)=\frac{d}{ds}\log S(s,\sigma)
+4\pi\vol(X)P_\sigma(s)+bs
\end{equation}
for some $b\in\C$.
Integrating this equality  gives
\begin{equation}\label{detform4}
\log S(s,\sigma)=-4\pi\vol(X)\int_0^s P_\sigma(r)\,dr
+\log\det_{\gr}\left(A(\sigma)+s^2\right)+\frac{b}{2} s^2+c
\end{equation}
for some $c\in\C$. In order to determine the constants $b$ and $c$ we take
$s\in\R$ and consider the asymptotic behavior of both sides of 
\eqref{detform4} as $s\to\infty$. 
By \eqref{logzeta} it follows that $\log S(s,\sigma)\to 0$ as $s\to\infty$.
Next consider the behavior of $\log\det_{\gr}\left(A(\sigma)+s^2\right)$ for
$s\in\R$ and $s\to\infty$. By \eqref{lauexp2} we have
\begin{equation}\label{lndet}
\log\det_{\gr}\left(A(\sigma)+s^2\right)
=-\frac{d}{dz}\left(z\zeta(z,s)\right)\big|_{z=0}.
\end{equation}
Denote the first term on the right hand side of \eqref{virtual8} by 
$I(t,\sigma)$ and the second by $H(t,\sigma)$.
Let $\Re(z)>3/2$. Then by \eqref{virtual8} and \eqref{dzeta} we get
\begin{equation}\label{dzeta1}
\begin{split}
\zeta(z,s)=\int_0^\infty e^{-ts^2} I(t,\sigma)t^{z-1}\;dt
+\int_0^\infty e^{-ts^2} H(t,\sigma)t^{z-1}\;dt.
\end{split}
\end{equation}
It follows from the definition of $H(t,\sigma)$ that the integral 
$\int_0^\infty e^{-ts^2} H(t,\sigma)t^{z-1}\;dt$ is an entire function of $z\in\C$
and for every compact subset $\omega\subset\C$ there exist $C,c>0$ such that
\begin{equation}\label{rapdecay}
\left|\frac{d}{dz}\int_0^\infty e^{-ts^2} H(t,\sigma)t^{z-1}\;dt\right|\le 
C \;e^{-cs^2}\quad z\in\omega,\;s\ge 0.
\end{equation}
To deal with the first integral on the right hand side of \eqref{dzeta1}, 
we note that
\begin{equation}\label{integr}
\int_0^\infty e^{-ts^2}\left(\int_\R e^{-t\lambda^2}\lambda^{2j}\;d\lambda\right)
dt=\Gamma(j+1/2)\Gamma(-j-1/2+z)s^{2j-2z+1}.
\end{equation}
By \eqref{planch2} the Plancherel polynomial $P_\sigma(z)$ is of the form
$P_\sigma(z)=a_1+a_2 z^2$. Using the definition of $I(t,\sigma)$ and
\eqref{integr}, we get
\begin{equation}
\begin{split}
\frac{d}{dz}\biggl(z\int_0^\infty e^{-ts^2} I(t,\sigma)t^{z-1}\;dt\biggr)
\bigg|_{z=0}&=-4\pi\vol(X)\left(a_1 s-\frac{a_2}{3}s^3\right)\\
&=-4\pi\vol(X)\int_0^s P_\sigma(r)\;dr.
\end{split}
\end{equation}
Together with \eqref{lndet}, \eqref{dzeta1}, and \eqref{rapdecay} we obtain
\[
\log\det_{\gr}\left(A(\sigma)+s^2\right)=4\pi\vol(X)\int_0^\infty P_\sigma(r)\;dr
+O(e^{-cs^2})
\]
for $s\in\R$, $s\to\infty$.
This implies that the constants $b$ and $c$ in \eqref{detform4} are zero.
Exponentiating \eqref{detform4}, we get \eqref{seldet1}. The proof of
\eqref{seldet2} is similar.
\end{proof}

\noindent
{\it Remark.}
From the statement of Theorem 3.19 in \cite{BO} it is not apparent that the 
determinant is the graded determinant. However, it is the general
understanding in \cite{BO} that the trace of a trace class operator on a graded
bundle is the super trace corresponding to the grading (see \cite[p. 29]{BO}).
Consequently regularized determinants of elliptic operators on graded bundles
are always understood in \cite{BO} as graded determinants.

\smallskip
Now let $\tau$ be an irreducible, finite-dimensional representation of $G$ 
with highest weight
$\Lambda_\tau=(m,n)$. For $w\in W_G$ let $\sigma_{\tau,w}\in\hat M$ and
$\lambda_{\tau,w}$ be defined by \eqref{wcharacter}. Let
\begin{equation}\label{casimir5}
\Delta(w)=\bigoplus_{\substack{\nu\\m_\nu(\sigma_{\tau,w})\neq0}}A_\nu+\tau(\Omega).
\end{equation}
This is an elliptic operator acting on $C^\infty(X,E(\sigma_{\tau,w}))$. 
Using \eqref{characterw1}, an explicite computation shows that for all 
$w\in W_G$ we have 
\begin{equation}\label{casimir2}
\lambda_{\tau,w}^2+c(\sigma_{\tau,w})=\frac{1}{2}\left(m(m+2)+n(n+2)\right)
=\tau(\Omega).
\end{equation}
Using \eqref{casimir2}, 
and \eqref{casimir5}, it follows that
\begin{equation}\label{casimir4}
A(\sigma_{\tau,w})+\lambda_{\tau,w}^2=\Delta(w)
\end{equation}
as operators on $C^\infty(X,E(\sigma_{\tau,w}))$. Then it follows from 
\eqref{seldet1}  that
\[
S(s-\lambda_{\tau,w};\sigma_{\tau,w})=\det_\gr(s^2-2\lambda_{\tau,w}s+\Delta(w))
\exp\left(-4\pi\vol(X)\int_0^{s-\lambda_{\tau,w}}P_{\sigma_{\tau,w}}(r)\,
dr\right),
\]
if $\sigma_{\tau,w}\neq\sigma_0$. If $\sigma_{\tau,w}=\sigma_0$, we use
\eqref{seldet2}, which leads to a similar formula
\begin{prop}
Let $\tau_\theta\ncong\tau$. There is a constant $c=c(\tau)$ such that
\begin{equation}\label{ruelledet1}
R_\tau(s)R_{\tau_\theta}(s)=e^{c\vol(X)s}\prod_{w\in W_G}
\det_\gr(s^2-2\lambda_{\tau,w}s+\Delta(w))^{(-1)^{\ell(w)+1}}.
\end{equation}
\end{prop}
\begin{proof}
By assumption we have $\tau=\tau_{m,n}$ with
$m\neq n$. It follows from \eqref{characterw1} that $\sigma_{\tau,w}\ncong\sigma_0$ for
all $w\in W_G$. Put
\begin{equation}\label{func}
F(s)=\sum_{w\in W_G}(-1)^{\ell(w)+1}\int_0^{s-\lambda_{\tau,w}}
P_{\sigma_{\tau,w}}(r)\;dr.
\end{equation}
Then it follows from \eqref{ruellesymm} and \eqref{seldet1} that
\begin{equation}\label{ruelledet2}
R_\tau(s)R_{\tau_\theta}(s)=e^{-4\pi\vol(X)F(s)}\prod_{w\in W_G}
\det_\gr\left(s^2-2s\lambda_{\tau,w}+\Delta(w)\right)^{(-1)^{\ell(w)+1}}.
\end{equation}
Using \eqref{planch2} and \eqref{characterw1}, an explicite computation gives
\[
F(s)=-\pi^{-2}(m+1)(n+1)s.
\]
\end{proof}
Now we consider the case $\tau_\theta=\tau$. Then $\tau=\tau_{m,m}$ for some 
$m\in\N_0$. 
\begin{prop}
Let $m\in\N_0$. There exists a constant $c=c(m)$ such that
\begin{equation}\label{ruelledet3}
R_{\tau_{m,m}}(s)=e^{c\vol(X)s}\cdot\frac{\det\left((s+m+1)^2-1+\Delta\right)
\det\left((s-m-1)^2-1+\Delta\right)}
{\det_{\gr}\left(s^2+A(\sigma_{2m+2})\right)}.
\end{equation}
\end{prop}
\begin{proof}
Put
\[
F_m(s)=\frac{1}{6\pi}\left((s+m+1)^3+(s-m-1)^3\right)
+4\pi\int_0^sP_{\sigma_{2m+2}}(r)
\,dr.
\]
Using \eqref{ruelle5}, \eqref{seldet1} and \eqref{seldet2}, it follows that
\[
R_{\tau_{m,m}}(s)=e^{\vol(X)F_m(s)}
\cdot\frac{\det\left((s+m+1)^2-1+\Delta\right)
\det\left((s-m-1)^2-1+\Delta\right)}
{\det_{\gr}\left(s^2+A(\sigma_{2m+2})\right)}.
\]
Using \eqref{planch2}, it follows that $F_m(s)=2\pi^{-1}(m+1)^2 s$.
\end{proof}

\section{Proof of Theorem \ref{theo-ruelle}}\label{heattrace}
\setcounter{equation}{0}

Since \cite{Wo} has not been published yet, we include a proof of Theorem 
\ref{theo-ruelle}. 
Let $\tau\colon G\to\GL(V_\tau)$ be an irreducible finite-dimensional 
representation with associated flat bundle $E_\tau$ equipped with an admissible
metric. Let $\Delta_p(\tau)$ be the Laplacian on $E_\tau$-valued $p$-forms. 
Let 
\begin{equation}\label{anator2}
K(t,\tau):=\sum_{p=1}^3(-1)^p p\Tr\left(e^{-t\Delta_p(\tau)}\right)
\end{equation}
and
\[
q(\tau)=\sum_{p=1}^3(-1)^p p\dim\ker H^p(X,E_\tau).
\]
Then by definition of the analytic torsion we have 
\begin{equation}\label{anator11}
\log T_X(\tau)=\frac{1}{2}\frac{d}{ds}\left(\frac{1}{\Gamma(s)}\int_0^\infty 
\left(K(t,\tau)-q(\tau)\right) t^{s-1}\;dt\right)\bigg|_{s=0},
\end{equation}
where the right hand side is defined near $s=0$ by analytic continuation of the
Mellin transform. The first step of the proof is to apply the trace formula to 
express 
$K(t,\tau)$ in terms of the length of closed geodesics. This is the basis 
for the relation between analytic torsion and the twisted Ruelle zeta 
function. 

Let $\pg$ be the orthogonal complement of $\kf$ in $\gf$ with respect to the
Killing form.  Let $x_0=eK$. Recall that there is a canonical  isomorphism 
$T_{x_0}(G/K)\cong\pg$.
Let $R_\Gamma$ denote the right regular representation 
of $G$ on $L^2(\Gamma\bs G)$ (resp. $C^\infty(\Gamma\bs G)$). 
Using \eqref{iso}, we get a canonical 
isomorphism
\begin{equation}\label{twistforms}
\Lambda^p(X,E_\tau)\cong\left(C^\infty(\Gamma\bs G)\otimes \Lambda^p\pg^*
\otimes V_\tau\right)^K,
\end{equation}
where $K$ acts by $k\in K\mapsto R_\Gamma(k)\otimes \Lambda^p\Ad_\pg^*(k)\otimes
\tau(k)$. There is a similar isomorphism for the space $L^2\Lambda^p(X,E_\tau)$
of $L^2$-sections of $\Lambda^pT^*X\otimes E_\tau$. 
With respect to the isomorphism \eqref{twistforms}, we have the following 
generalization of Kuga's lemma 
\begin{equation}\label{kuga1}
\Delta_p(\tau)=-R_\Gamma(\Omega)\otimes\Id+\tau(\Omega)\Id,
\end{equation}
(see \cite[(6.9)]{MM}), where $\Omega$ is the Casimir element and 
$\tau(\Omega)$ is the Casimir eigenvalue of $\tau$. 

Let $\tilde\Delta_p(\tau)$ be the lift of $\Delta_p(\tau)$ to the universal
covering $\tilde X=G/K$. Let $e^{-t\tilde\Delta_p(\tau)}$, $t>0$,
 be the corresponding heat semigroup. This is
a smoothing operator on 
\[
L^2\Lambda^p(\tilde X;\tilde E_\tau)\cong (L^2(G)\otimes\Lambda^p\mathfrak 
p^*\otimes V_\tau)^K,
\]
which commutes with the action of $G$. Therefore, it is of the form
\begin{displaymath}
 \left( e^{-t\tilde\Delta_p(\tau)}\phi\right)(g)=\int_G  
H^{\tau,p}_t(g^{-1}g')\phi(g') \;dg',\quad 
\phi\in(L^2(G)\otimes\Lambda^p\mathfrak p^*\otimes V_\tau)^K,\;\;g\in G,
\end{displaymath}
where the kernel
$H^{\tau,p}_t\colon G\to\End(\Lambda^p\mathfrak p^*\otimes
V_\tau)$ 
belongs to $C^\infty\cap L^2$ and  satisfies the covariance property
\begin{equation}\label{covar}
H^{\tau,p}_t(k^{-1}gk')=\nu_p(\tau)(k)^{-1} H^{\tau,p}_t(g)\nu_p(\tau)(k'),
\end{equation}
with respect to the representation 
\begin{equation}
\nu_p(\tau):=\Lambda^p\Ad_K^*\otimes \tau\colon K\to 
\GL(\Lambda^p\pg^*\otimes V_\tau).
\end{equation}
Moreover, for all $q>0$ we have
\begin{equation}\label{schwartz1}
H^{\tau,p}_t \in (\mathcal{C}^q(G)\otimes\End(\Lambda^p\mathfrak 
p^*\otimes V_\tau))^{K\times K},
\end{equation}
where $\Co^q(G)$ denotes Harish-Cahndra's $L^p$-Schwartz space. 
The proof is similar to the proof of Proposition 2.4 in \cite{BM}. Let
\[ 
h^{\tau,p}_t(g)=\tr H^{\tau,p}_t(g). 
\] 
Repeating the arguments which we used to prove
\eqref{trace2}, we get
\begin{equation}
\Tr\left(e^{-t\Delta_p(\tau)}\right)=\Tr R_\Gamma(h^{\tau,p}_t).
\end{equation}
Put
\begin{equation}\label{alter}
  k_t^\tau=\sum_{p=1}^3(-1)^pp\, h^{\tau,p}_t.
\end{equation}
By \eqref{anator2} we have
\[
K(t,\tau)=\Tr R_\Gamma(k_t^\tau). 
\]
We can now apply the Selberg trace formula \cite{Wa1}. Let the notation be as 
in \eqref{anator3}. Then we get
\begin{equation}\label{anator3a}
\begin{split}    
K(t,\tau)&=\OP{Vol}(X)k^\tau_t(e)\\
  &\mspace{30mu} +\frac{1}{2\pi}\sum_{[\gamma]\neq e}
\frac{\ell(\gamma)}{n_\Gamma(\gamma)D(\gamma)}\sum_{n\in\Z}
\overline{\sigma_n(m_\gamma)}\int_\R\Theta_{n,\lambda}(k^\tau_t)
 e^{-i\ell(\gamma)\lambda}\;d\lambda,
\end{split}
\end{equation}
where the notation is the same as in \eqref{anator3}. 
The characters $\Theta_{n,\lambda}(k^\tau_t)$ can be computed in the same was 
as in section \ref{sec-bochlapl}.
Let $\pi$ be a unitary representation of $G$ on 
a Hilbert space $\H_\pi$. Set
\[
\tilde \pi(H^{\tau,p}_t)=\int_G \pi(g)\otimes H^{\tau,p}_t(g)\,dg.
\]
This defines a bounded operator on 
$\H_\pi\otimes \Lambda^p\mathfrak p^*\otimes V_\tau$. As in 
\cite[pp. 160-161]{BM} it follows from 
\eqref{covar} that relative to the splitting
\[
\H_\pi\otimes \Lambda^p\mathfrak p^*\otimes V_\tau=
\left(\H_\pi\otimes \Lambda^p\mathfrak p^*\otimes V_\tau\right)^K\oplus
\left[\left(\H_\pi\otimes \Lambda^p\mathfrak p^*\otimes V_\tau\right)^K
\right]^\perp,
\]
$\tilde \pi(H^{\tau,p}_t)$ has the form
\[
\tilde \pi(H^{\tau,p}_t)=\begin{pmatrix}\pi(H^{\tau,p}_t)& 0\\ 0& 0
\end{pmatrix}
\]
with $\pi(H^{\tau,p}_t)$ acting on $\left(\H_\pi\otimes 
\Lambda^p\mathfrak p^*\otimes V_\tau\right)^K$. Using \eqref{kuga1} it follows 
as in \cite[Corollary 2.2]{BM} that
\[
\pi(H^{\tau,p}_t)=e^{t(\pi(\Omega)-\tau(\Omega))}\Id
\]
on $\left(\H_\pi\otimes \Lambda^p\mathfrak p^*\otimes V_\tau\right)^K$.
As in \eqref{equtrace} we get
\begin{equation}\label{equtrace1}
\tr\pi(H^{\tau,p}_t)=\tr\pi(h^{\tau,p}_t).
\end{equation}
Now let $\pi$ be a unitary principal series representation $\pi_{n,\lambda}$. 
Using \eqref{equtrace1} and \eqref{ucasimir} we get 
\begin{equation}\label{fourierinv}
\Theta_{n,\lambda}(h^{\tau,p}_t)=e^{-t(\lambda^2+1-n^2/4+\tau(\Omega))}
  \dim\left(\H_{n}\otimes\Lambda^p\mathfrak p^*\otimes V_\tau\right)^K.
\end{equation}
Denote by $\C_n$ the $M$-module defined by $\sigma_n$.  By Frobenius 
reciprocity \cite[p. 208]{Kn} we have 
\[
\dim\left(\H_{\xi,\lambda}\otimes\Lambda^p\mathfrak p^*\otimes V_\tau\right)^K
=\dim\left(\C_n\otimes\Lambda^p\mathfrak p^*\otimes V_\tau\right)^M.
\]
and by \eqref{alter} we get
\[
\Theta_{n,\lambda}(k_t^\tau)=e^{-t(\lambda^2+1-n^2/4+\tau(\Omega))}
\sum_{p=1}^3(-1)^p p\dim\left(\C_n\otimes\Lambda^p\mathfrak p^*\otimes 
V_\tau\right)^M.
\]
Choose an orthonormal basis of $\pg$ as in \cite[p.9]{Mil}. Using this basis it
follows that as $M$-modules,  $\pg$ and $\af\oplus\nf$ are equivalent. Thus
we get 
\begin{equation}
\sum_{p=1}^{3}(-1)^p p\,\Lambda^p\mathfrak  p^*
=\sum_{p=1}^{3}(-1)^{p}p\left(\Lambda^p\mathfrak n^*
+\Lambda^{p-1}\mathfrak n^*\right)
=\sum_{p=0}^{2}(-1)^{p+1}\Lambda^p\mathfrak n^*.
\end{equation}
Therefore, the Fourier transformation of $k^\tau_t$ is given by
\begin{equation}\label{charac1}
\Theta_{n,\lambda}(k^\tau_t)=e^{-t(\lambda^2+1-n^2/4+\tau(\Omega))}
\sum_{p=0}^{2}(-1)^{p+1}\dim(\C_n\otimes\Lambda^p\mathfrak n^*\otimes V_\tau)^M.
\end{equation}
This formula can be simplified using the real version of Konstant's 
Bott-Borel-Weil theorem \cite{Si}. 
We apply Lemma \ref{thmkostant} to determine the $n\in\Z$ for which 
$\dim(\C_n\otimes\Lambda^p\mathfrak n^*\otimes V_\tau)^M\neq0$. We decompose the
characters on the right hand side of \eqref{kostant1} according to
\eqref{wcharacter}. 
Let $\sigma_{\tau,w}\in\hat M$ and $\lambda_{\tau,w}\in\frac{1}{2}\Z$ be defined
by \eqref{wcharacter}. 
Using \eqref{kostant1} and \eqref{casimir2}, we get
\begin{equation}\label{charac2}
\begin{split}
\sum_{n\in\Z}
\overline{\sigma_n(m_\gamma)}\int_\R\Theta_{n,\lambda}(k^\tau_t)
&e^{-i\ell(\gamma)\lambda}\,d\lambda\\
&=\sum_{w\in W_G}(-1)^{\ell(w)+1}
\sigma_{\tau,w}(m_\gamma)
e^{-t \lambda_{\tau,w}^2}\frac{e^{-\ell(\gamma)^2/(4t)}}{(4\pi t)^{1/2}}.
\end{split}
\end{equation}
Next we consider the contribution of the identity to \eqref{anator3a}. By
\eqref{schwartz1},  $k^\tau_t$ is in $\mathcal{C}^q(G)$ for
all $q>0$. Therefore we can  apply the Plancherel  formula for
$G$ (see \cite[Theorem 11.2]{Kn}).  With respect to the normalizations of Haar
measures used in \cite{Kn}, we have
\begin{displaymath}
  k^\tau_t(e)=\sum_{n\in\Z}\int_{\mathbb R}
\Theta_{n,\lambda}(k^\tau_t)P_{\sigma_n}(i\lambda)\,d\lambda,
\end{displaymath}
where $P_{\sigma_n}(z)$ is the Plancherel polynom \eqref{planch2}.
Repeating the arguments that led to \eqref{charac2}, we get
\begin{equation}
k_t^\tau(e)=\sum_{w\in W_G}(-1)^{\ell(w)+1} e^{-t \lambda_{\tau,w}^2}
\int_{\mathbb R} e^{-t\lambda^2}P_{\sigma_{\tau,w}}(i\lambda)\,d\lambda.
\end{equation}
Combined with  \eqref{anator3a} and \eqref{charac2}, we obtain
\begin{equation}\label{anator4}
\begin{split}
K(t,\tau)=\sum_{w\in W_G}(-1)^{\ell(w)+1} e^{-t \lambda_{\tau,w}^2}
\biggl(\vol(X)&\int_{\R} e^{-t\lambda^2}P_{\sigma_{\tau,w}}(i\lambda)\,d\lambda\\
&\mspace{14mu}+\sum_{\{\gamma\}\neq\{e\}}
\frac{\ell(\gamma)}{n_\Gamma(\gamma)}L(\gamma;\sigma_{\tau,w})
\frac{e^{-\ell(\gamma)^2/(4t)}}{(4\pi t)^{1/2}}\biggr),
\end{split}
\end{equation}   
where $L(\gamma,\sigma)$ is defined by 
\begin{equation}\label{lefschetz1}
L(\gamma,\sigma)=\frac{\sigma(m_\gamma)
e^{-\ell(\gamma)}}
{\det\left(\Id-\Ad(m_\gamma a_\gamma)_{\overline \nf}\right)}.
\end{equation}

Unfortunately, the constants 
$\lambda_{\tau,w}$ appearing in the exponential factors prevent us from 
applying the Mellin transform to this formula directly. This problem occurred
already in \cite{Fr}. To overcome this problem we use the auxiliary operators
introduced in section \ref{sec-bochlapl}.

Using \eqref{virtual11} and \eqref{casimir2}, it follows 
that for $w\in W_G$ we have
\begin{equation}\label{virtual5}
\begin{split}
e^{\tau(\Omega)t}K(t,\sigma_{\tau,w})=e^{-t\lambda_{\tau,w}^2}
\biggl(2&\vol(X)\int_\R e^{-t\lambda^2}
P_{\sigma_{\tau,w}}(i\lambda)\;d\lambda\\
&+\sum_{[\gamma]\neq 1}\frac{\ell(\gamma)}{n_\Gamma(\gamma)}
(L(\gamma;\sigma_{\tau,w})+L(\gamma;w_A(\sigma_{\tau,w})))\frac{e^{-\ell(\gamma)^2/(4t)}}{(4\pi t)^{1/2}}\biggr).
\end{split}
\end{equation}
Next observe that by \eqref{characterw1} there exists a decomposition 
\[
W_G=W_{0}\sqcup W_{1},
\]
with $|W_{i}|=2$, $i=1,2$, and a bijection $j\colon W_{0}\to W_{1}$ such
that for $w\in W_{0}$ we have
\[
\sigma_{\tau,j(w)}=w_A(\sigma_{\tau,w}),\quad \lambda_{\tau,j(w)}=
-\lambda_{\tau,w}.
\]
Hence by \eqref{anator4} and \eqref{virtual5} we get
\begin{equation}\label{anator5}
K(t,\tau)=\frac{1}{2}
\sum_{w\in W_G}(-1)^{\ell(w)+1}e^{\tau(\Omega)t}K(t;\sigma_{\tau,w}).
\end{equation}
This equality can be expressed in a slightly different way as follows. 
Denote by $\Tr_s$ the supertrace with respect to the grading of 
$E(\sigma_{\tau,w})$. Using the definition of $K(t,\sigma_{\tau,w})$ by
\eqref{virtual1} and the definition of $\Delta(w)$ by \eqref{casimir5}, we get
\begin{equation}\label{anator6}
K(t,\tau)=\frac{1}{2}
\sum_{w\in W_G}(-1)^{\ell(w)+1}\Tr_s\left(e^{-t\Delta(w)}\right).
\end{equation}

To continue we need to determine the location of the spectrum of the operators 
$A(\sigma)$.
\begin{lem}\label{spectrum}
For $\sigma\in\hat M$ we have $A(\sigma)\ge -1$. Moreover, if 
$k\notin\{0,\pm2\}$, then $A(\sigma_k)>-1$. 
\end{lem}
\begin{proof}
Let $\hat G$ denote the unitary dual of $G$. Let
\begin{equation}\label{regrepr}
L^2(\Gamma\bs G)=\widehat\bigoplus_{\pi\in\hat G} m_\Gamma(\pi)\H_\pi
\end{equation}
be the spectral decomposition of the right regular representation of $G$
on $L^2(\Gamma\bs G)$. Let $(\nu,V_\nu)$ be an irreducible unitary 
representation of $K$.  Then $L^2(X,E_\nu)\cong \left(L^2(\Gamma\bs G)\otimes 
V_\nu\right)^K $. Using \eqref{regrepr}, we get
\begin{equation}\label{specdecomp}
\left(L^2(\Gamma\bs G)\otimes V_\nu\right)^K =\widehat\bigoplus _{\pi\in\hat G} 
m_\Gamma(\pi)\left(\H_\pi\otimes V_\nu\right)^K.
\end{equation}
This decomposition corresponds to the spectral resolution of $A_\nu$ as follows.
Assume that $m_\Gamma(\pi)\dim(\H_\pi\otimes V_\nu)^K\neq 0$. Then 
$m_\Gamma(\pi)(\H_\pi\otimes V_\nu)^K$ is an eigenspace
of $A_\nu$ with eigenvalue $-\pi(\Omega)$. Note that $\hat G$ is the union
of the trivial representation, the unitary principal series $\pi_{k,\lambda}$
with $k\in\Z$ and $\lambda\in\R$, and the complementary series $\pi_{x}^c$
with $0<x<1$ \cite[Proposition 49]{KS}, \cite[Theorem 16.2]{Kn}
(where for the latter reference the different parametrization of the induced
representations has to be taken into account).  

First consider the principal series $\pi_{n,\lambda}$.
By Frobenius reciprocity \cite[p. 208]{Kn} we have for $l\in\N_0$
\begin{equation}\label{dimension}
\dim(\H_{\pi_{k,\lambda}}\otimes V_{\nu_l})^K=[\pi_{k,\lambda}^\vee|_K:\nu_l]=
[\nu_l|_M:\sigma_{-k}]=[\nu_l|_M:\sigma_{k}].
\end{equation}
By \eqref{restrict1} it follows that 
$(\H_{\pi_{k,\lambda}}\otimes V_{\nu_l})^K\neq0$ implies $l\ge k$.
Moreover, by \eqref{restrict2} it follows that $m_{\nu_l}(\sigma_m)\neq 0$ 
implies $m\ge l$.
Thus if $m_{\nu}(\sigma_m)\neq0$ and 
$(\H_{\pi_{k,\lambda}}\otimes V_{\nu})^K\neq0$, then we have $m\ge k$.
Hence if $\nu\in\hat K$ and $\sigma\in\hat M$ are such that 
$m_\nu(\sigma)\neq0$ and $(\H_{\pi_{k,\lambda}}\otimes V_{\nu})^K\neq0$, then
it follows from \eqref{ucasimir} and \eqref{cl} that
\begin{equation}\label{casimir9}
-\pi_{k,\lambda}(\Omega)+c(\sigma)\ge 0.
\end{equation}
Next consider the complementary series. By \eqref{cl} we have $c(\sigma)\ge 
-1$ for all $\sigma\in\hat M$. Since $0< x<1$,  it follows from 
\eqref{casimir8} that
\begin{equation}\label{casimir10}
-\pi_{x}^c(\Omega)+c(\sigma)> -1.
\end{equation}
for all $\sigma\in\hat M$. Finally, the trivial representation of $G$ 
occurs in \eqref{specdecomp} only if $\nu$ is the trivial representation 
$\nu_0$. Moreover, by \eqref{restrict2} we have $m_{\nu_0}(\sigma_l)\neq0$,
only if $l=0$ or $l=2$. Thus by \eqref{casimir9}, \eqref{casimir10}, 
and the definition of $A(\sigma)$ by \eqref{aoperator}, the statement of 
the Lemma follows.
\end{proof}
We apply this lemma to study the kernel of the operator $\Delta(w)$, $w\in W_G$,
 which is defined by \eqref{casimir5}.

\begin{lem}\label{kernel}
Let $\tau$ be an irreducible, finite-dimensional representation of $G$. 
Assume that $\tau_\theta\ncong\tau$. Then $\ker\Delta(w)=\{0\}$ for all 
$w\in W_G$.
\end{lem}
\begin{proof}
Let $\tau=\tau_{m,n}$ with $m\neq n$. We use \eqref{casimir4} to express 
$\Delta(w)$ in terms of $A(\sigma_{\tau,w})$. By \eqref{characterw1} we
have $\lambda_{\tau,w}\in\frac{1}{2}\Z\setminus\{0\}$ for all $w\in W_G$. 
If $|\lambda_{\tau,w}|>1$, it follows from \eqref{casimir4} that
$\Delta(w)>0$. It remains to consider the cases $\lambda_{\tau,w}=\pm 1$
and $\lambda_{\tau,w}=\pm 1/2$. In the first case we have $|m-n|=2$. Then it
follows from \eqref{characterw1} that $\sigma_{\tau,w}=\sigma_{2l}$ with 
$|l|\ge 2$. By Lemma \ref{spectrum} we get $\Delta(w)>0$. In the second case
we have $|m-n|=1$. By \eqref{characterw1} it follows that $\sigma_{\tau,w}
=\sigma_{2l+1}$ for some $l\in\Z$. Let $\nu\in\hat K$ such that 
$m_\nu(\sigma_{2l+1})\neq0$. By \eqref{restrict2} there exists $p\in\N_0$ such
that $\nu=\nu_{2p+1}$. Since $\pi^c_x$ is induced from the trivial
 representation and $[\nu_{2p+1}|_M:\sigma_0]=0$, Frobenius reciprocity 
\cite[p. 208]{Kn} implies
\begin{equation}\label{dimension2}
\dim(\H_{\pi^c_x}\otimes V_{\nu_{2p+1}})^K=[\pi^c_x|_K:\nu_{2p+1}]=
[\nu_{2p+1}|_M:\sigma_0]=0.
\end{equation}
Thus in this case the complementary series does not occur
in \eqref{specdecomp}. Also the trivial representation does not occur.
By \eqref{casimir9} it follows that $A(\sigma_{\tau,w})\ge 0$. Using
\eqref{casimir4} we get $\Delta(w)>0$. 
\end{proof}

Now we can turn to the proof of Theorem \ref{theo-ruelle}. First assume that 
$\tau\ncong\tau_\theta$. Then it follows from
\cite[Chapt. VII, Theorem 6.7]{BW} that $H^*(X,E_\tau)=0$. Hence $\Delta_p(\tau)
>0$ for all $p$, $0\le p\le 3$. By Lemma \ref{kernel} we also have 
$\Delta(w)>0$, $w\in W_G$. Hence $K(t,\tau)$ and $\Tr(e^{-t\Delta(w)})$, 
$w\in W_G$, are exponentially decreasing as $t\to\infty$. Therefore we can
take the Mellin transform of both sides of \eqref{anator6} and we get
\[
\frac{1}{\Gamma(s)}\int_0^\infty K(t,\tau)t^{s-1}\;dt=
\frac{1}{2}\sum_{w\in W_G}(-1)^{\ell(w)+1}\frac{1}{\Gamma(s)}\int_0^\infty
\Tr_s\left(e^{-t\Delta(w)}\right) t^{s-1}\,dt,
\]
which holds for $\Re(s)>3/2$. After analytic continuation we compare the 
derivatives at $s=0$ of both sides. Using \eqref{anator11} we get
\begin{equation}\label{anatorvirt}
T_X(\tau)^4=\prod_{p=1}^3\det\left(\Delta_p(\tau)\right)^{2(-1)^{p+1}p}=
\prod_{w\in W_G}\det_{\gr}\left(\Delta(w)\right)^{(-1)^{\ell(w)+1}}.
\end{equation}
Now we use the determinant formula \eqref{seldet1} to relate the right hand
side to the value at zero of the Ruelle zeta function.
Since $\Delta(w)>0$,  it follows that
$\det_{\gr}(s^2-2s\lambda_{\tau,w}+\Delta(w))$ is regular at $s=0$ and
its value at $s=0$ is equal to $\det_{\gr}\left(\Delta(w)\right)\neq0$. Hence
\begin{equation}\label{value1}
\lim_{s\to 0}\prod_{w\in W_G}
\det_{\gr}\left(s^2-2s\lambda_{\tau,w}+\Delta(w)\right)^{(-1)^{\ell(w)+1}}
=\prod_{w\in W_G}\det_{\gr}\left(\Delta(w)\right)^{(-1)^{\ell(w)+1}}.
\end{equation}
By \eqref{ruelledet2} it follows that $R_\tau(s)R_{\tau_\theta}(s)$ is regular
at zero. Now observe that $\overline\tau=\tau_\theta$. Furthermore by 
\eqref{ruelle1} we have
\[
\overline{R_\tau(s)}=R_{\tau_\theta}(\overline s).
\]
This implies that $R_\tau(s)$ is regular at $s=0$ and
\begin{equation}\label{value2}
|R_\tau(0)|^2=\prod_{w\in W_G}\det_{\gr}\left(\Delta(w)\right)^{(-1)^{\ell(w)+1}}.
\end{equation}
Combining \eqref{anatorvirt} and \eqref{value2}, the first statement of 
Theorem \ref{theo-ruelle} follows.

Next assume that $\tau_\theta=\tau$. Then there exists $m\in\N_0$ such that
$\tau=\tau_{m,m}$. We use 
\eqref{ruelledet3} to determine the order of $R_{\tau}(s)$ at $s=0$. For
$m\in\N_0$ let
\begin{equation}
h_m=\dim_{\gr}\ker(A(\sigma_{2m+2})),
\end{equation}
where $\dim_{\gr}$ denotes the graded dimension of a graded vector space, 
i.e., if $V=V^+\oplus V^-$ is 
a graded finite-dimensional vector space, then $\dim_{\gr}V=\dim V^+-\dim V^-$.
Assume that $m\ge 1$. Then $\det\left((s\pm(m+1))^2-1+\Delta\right)$ is 
regular and nonzero at $s=0$. Furthermore, 
$\det_{\gr}\left(s^2+A(\sigma_{2m+2})\right)$ has order $2h_m$ at $s=0$ and
\begin{equation}\label{ruelleord1}
\lim_{s\to 0}s^{-2h_m}\det_{\gr}\left(s^2+
A(\sigma_{2m+2})\right)=\det_{\gr}\left(A(\sigma_{2m+2})\right).
\end{equation}
By \eqref{ruelledet3}, it follows that $R_{\tau}(s)$ has order $-2h_m$
at $s=0$ and we have
\begin{equation}\label{laurent1}
\lim_{s\to0}s^{2h_m}R_{\tau_{m,m}}(s)=
\frac{\det\left((m+1)^2-1+\Delta\right)^2}
{\det_{\gr}\left(A(\sigma_{2m+2})\right)}.
\end{equation}
On the other hand, using \eqref{characterw2}, it follows from 
\eqref{anator6} that
\begin{equation}\label{anator9}
K(t,\tau)=\Tr_s\left(e^{-tA(\sigma_{2m+2})}\right)-
2e^{-t\left((m+1)^2-1\right)}\Tr\left(e^{-t\Delta}\right).
\end{equation}
Taking the limit $t\to\infty$ of both sides of this equality, we get
\begin{equation}\label{ruelleord2}
h_m=\sum_{p=1}^3(-1)^p p\dim\left(\ker\Delta_p(\tau)\right).
\end{equation}
Moreover \eqref{anator9} also implies
\[
T_X(\tau)^2=\prod_{p=1}^3\det\left(\Delta_p(\tau)\right)^{(-1)^{p+1}p}
=\frac{\det\left((m+1)^2-1+\Delta\right)^2}
{\det_{\gr}\left(A(\sigma_{2m+2})\right)}.
\]
Combining this equality with \eqref{laurent1} and \eqref{ruelleord2}, we 
obtain the second statement of Theorem \eqref{theo-ruelle}.

\section{Proof of Theorem \ref{asymptotic}}\label{mainth}

We are now ready to prove our main result.
We consider the representation 
$\tau_m$. Using \eqref{restrict5}, it follows from (\ref{decompo1})that 
\begin{equation}\label{prod1}
R_{\tau_m}(s)=\prod_{k=0}^m R\left(s-\left(m/2-k\right),\sigma_{m-2k}\right).
\end{equation}
We distinguish the cases where $m$ is odd and even. Let $m\ge3$. Then  we get
\begin{equation}\label{prod2}
\begin{split}
R_{\tau_{2m}}(s)&=\prod_{k=0}^{4} R(s-(2-k),\sigma_{4-2k})
\prod_{k=3}^{m} R(s-k,\sigma_{2k})R(s+k,\sigma_{-2k})\\
&=R_{\tau_4}(s)\prod_{k=3}^{m} R(s-k,\sigma_{2k})R(s+k,\sigma_{-2k}).
\end{split}
\end{equation}
Similarly, for $m\ge 2$ we get
\begin{equation}\label{prod3}
R_{\tau_{2m+1}}(s)=
R_{\tau_3}(s)\prod_{k=2}^m R(s-k-1/2,\sigma_{2k+1})R(s+k+1/2,\sigma_{-(2k+1)}).
\end{equation}
Now recall that by Proposition \ref{prop-ruelle}, 1), each $R(s,\sigma_l)$, 
$l\in\Z$, is regular  in the half-plane $\Re(s)>2$ and does not vanish in this
half-plane. By the functional equation \eqref{functequ} the same holds in the 
half-plane $\Re(s)<2$. Therefore the products on the right hand side of 
\eqref{prod2}
and \eqref{prod3} are regular at $s=0$. Furthermore it follows from 
\eqref{ruellezeta} that 
\begin{equation}\label{modulus}
|R(s,\sigma_l)|=|R(\overline s,\sigma_{-l})|.
\end{equation}
Using  \eqref{prod2}, \eqref{prod3} and Theorem \eqref{theo-ruelle} we get
\begin{equation}\label{tors1}
T_X(\tau_{2m})^2=T_X(\tau_4)^2\prod_{k=3}^{m} |R(k,\sigma_{2k})|\cdot 
|R(-k,\sigma_{2k})|,\quad m\ge 3.
\end{equation}
and
\begin{equation}\label{tors2}
T_X(\tau_{2m+1})^2=T_X(\tau_3)^2\prod_{k=2}^{m} |R(k+1/2,\sigma_{2k+1})|\cdot 
|R(-k-1/2,\sigma_{2k+1})|,\quad m\ge 2.
\end{equation}
By the functional equation \eqref{functequ} and \eqref{modulus} we get
\[
|R(-k,\sigma_{2k})|=
\exp\left(-\frac{4}{\pi}\vol(\Gamma\bs\bH^3)k\right)|R(k,\sigma_{2k})|
\]
Together with \eqref{tors1} this leads to 
\begin{equation}\label{tors3}
T_X(\tau_{2m})=T_X(\tau_4)
\prod_{k=3}^m\exp\left(-\frac{2}{\pi}\vol(\Gamma\bs\bH^3)k\right)
 |R(k,\sigma_{2k})|.
\end{equation}
Similarly
\begin{equation}\label{tors4}
T_X(\tau_{2m+1})=T_X(\tau_3)\prod_{k=2}^m
\exp\left(-\frac{2}{\pi}\vol(\Gamma\bs\bH^3)(k+1/2)\right)
 |R(k+1/2,\sigma_{2k+1})|.
\end{equation}

To continue we need the following estimation.
\begin{lem}\label{bound1} 
There exists $C>0$ such that for all $m\in\N$, $m\ge 3$, we have
\[
\sum_{k=3}^m\big|\log |R(k,\sigma_{2k})|\big|\le C,\quad
\sum_{k=2}^m\big|\log |R(k+1/2,\sigma_{2k+1})|\big|\le C.
\]
\end{lem}
\begin{proof} We consider the first case.
Since $|\sigma_{2k}(m_\gamma)|=1$, we have
\[
1-e^{-k\ell(\gamma)}\le \big|1-\sigma_{2k}(m_\gamma)e^{-k\ell(\gamma)}\big|\le
1+e^{-k\ell(\gamma)}.
\]
Let $k\ge 3$. Using that the infinite product (\ref{ruellezeta}) is absolutely 
convergent for $\Re(s)> 2$, we get
\[
\sum_{\substack{[\gamma]\not=e\\\pr}}\log\big(1-e^{-k\ell(\gamma)}\big)\le
\log|R(k,\sigma_{2k})|\le 
\sum_{\substack{[\gamma]\not=e\\\pr}}\log\big(1+e^{-k\ell(\gamma)}\big),
\]
which implies
\begin{equation}\label{bound3}
\big|\log|R(k,\sigma_{2k})|\big|\le \sum_{\substack{[\gamma]\not=e\\\pr}}
\sum_{n=1}^\infty\frac{1}{n}e^{-nk\ell(\gamma)}.
\end{equation}
Let $\delta=\inf\{\ell(\gamma)\colon \gamma\in\Gamma\setminus\{e\}\}$.
Put $C_1=(1-e^{-\delta})^{-1}$. 
Using \eqref{bound3}  we get
\[
\begin{split}
\sum_{k=3}^m\big|\log |R(k,\sigma_{2k})|\big|&\le\sum_{n=1}^\infty \frac{1}{n}
\sum_{\substack{[\gamma]\not=e\\\pr}}\sum_{k=3}^m e^{-\ell(\gamma)nk}\\
&=\sum_{n=1}^\infty \frac{1}{n}\sum_{\substack{[\gamma]\not=e\\\pr}}
\left(\frac{1-e^{-n(m+1)\ell(\gamma)}}{1-e^{-n\ell(\gamma)}}-
\left(1+e^{-n\ell(\gamma)}+e^{-2n\ell(\gamma)}\right)\right)\\
&=\sum_{n=1}^\infty\frac{1}{n} \sum_{\substack{[\gamma]\not=e\\\pr}}
\frac{e^{-3n\ell(\gamma)}-e^{-(m+1)n\ell(\gamma)}}{1-e^{-n\ell(\gamma)}}\\
&\le C_1 \sum_{\substack{[\gamma]\not=e\\\pr}}\sum_{n=1}^\infty
\frac{e^{-3n\ell(\gamma)}}{n}\\ 
&=C_1\log R(3,\sigma_0)^{-1}=C.
\end{split}
\]
The other case is similar.
\end{proof}
Taking the logarithm of both sides of \eqref{tors3} and \eqref{tors4}, 
respectively, we obtain
\[
\begin{split}
\log T_X(\tau_{2m})=\log T_X(\tau_4)+\sum_{k=3}^m\log |R_{2k}(k)|
-\frac{1}{\pi}\vol(\Gamma\bs\bH^3)\left(m(m+1)-6\right).
\end{split}
\]
and 
\[
\begin{split}
\log T_X(\tau_{2m+1})=\log T_X(\tau_3)+\sum_{k=2}^m\log \Bigg|R_{2k+1}
\left(k+\frac{1}{2}\right)\Bigg|
-\frac{1}{\pi}\vol(\Gamma\bs\bH^3)(m(m+2)-3).
\end{split}
\]
Applying Lemma \ref{bound1} we get
\[
-\log T_X(\tau_m)=\frac{1}{4\pi}\vol(\Gamma\bs\bH^3)m^2 + O(m)
\]
as $m\to\infty$.  This completes the proof of Theorem \ref{asymptotic}.


\begin{thebibliography}{MMM}
\bibitem[BM]{BM} D. Barbasch, H. Moscovici, {\it $L^2$-index and the trace 
formula}, J. Funct. Analysis {\bf 53} (1983), 151--201.
\bibitem[Be]{Be} R. Berndt, {\it Representations of linear groups. An 
introduction based on examples from physics and number theory.} Vieweg, 
Wiesbaden, 2007. 
\bibitem[BV]{BV} J.-M. Bismut, E. Vasserot, {\it The asymptotics of the 
Ray-Singer analytic torsion of the symmetric powers of a positive vector 
bundle.}  Ann. Inst. Fourier (Grenoble)  {\bf 40}  
(1990),  no. 4, 835--848 (1991).
\bibitem[BW]{BW} A. Borel, N. Wallach, {\it Continuous cohomology, discrete 
subgroups, and representations of reductive groups}, Second edition. 
Mathematical Surveys and Monographs, 67. Amer. Math. Soc., 
Providence, RI, 2000.
\bibitem[Br]{Br} U. Br\"ocker, {\it Die Ruellesche Zetafunktion f\"ur 
$G$-induzierte Anosov-Fl\"usse}, Ph.D. thesis, Humboldt-Universit\"at Berlin,
Berlin, 1998.
\bibitem[BO]{BO} U. Bunke and M. Olbirch, 
{\it Selberg zeta and theta functions}, A differential operator approach, 
Akademie Verlag, Berlin, 1995.
\bibitem[Ch]{Ch} T.A. Chapman, {\it Topological invariance of Whitehaed 
torsion}, Amer. J. Math. {\bf 96} (1974), 488 -- 497.
\bibitem[EGM]{EGM} J. Elstrodt, F. Grunewald, J. Mennicke, {\it Groups acting 
on hyperbolic space. Harmonic analysis and number theory.} Springer Monographs 
in Mathematics. Springer-Verlag, Berlin, 1998.
\bibitem[Fr]{Fr} D. Fried, {\it Analytic torsion and closed geodesics on
hyperbolic manifolds}, Invent. math. {\bf 84} (1986), 523--540.
\bibitem[Fr2]{Fr2} D. Fried, {\it Meromorphic zeta functions of analytic 
flows}, Commun. Math. Phys. {\bf 174} (1995), 161 - 190.
\bibitem[Kn]{Kn} A.W. Knapp, {\it Representation theory of semisimple groups},
Princeton University Press, Princeton and Oxford, 2001.
\bibitem[KS]{KS} A.W. Knapp and E.M. Stein, {\it Intertwining operators for 
semisimple Lie groups}, Annals of Math. {\bf 93} (1971), 489--578.
\bibitem[Ko]{Ko} B. Kostant, {\it Lie algebra cohomology and the 
generalized Borel-Weil theorem.}  Ann. of Math. (2)  {\bf 74}  (1961) 329--387.
\bibitem[MM]{MM} Matsushima, Murakami, {\it On vector bundle valued harmonic 
forms and automorphic forms on symmetric riemannian manifolds},  
Ann. of Math. {\bf 78} (1963), 365--416. 
\bibitem[Mia]{Mia} R.J. Miatello, {\it The Minakshisundaram-Pleijel 
coefficients for the vector-valued heat kernel on compact locally symmetric 
spaces of negative curvature.}  Trans. Amer. Math. Soc. {\bf 260} (1980), 
1--33. 
\bibitem[Mil]{Mil} J.J. Millson, {\it Closed geodesics and the 
$\eta$-invariant}, Annals of Math. {\bf 108} (1978), 1--39.
\bibitem[Mn]{Mn} J. Milnor, {\it Whitehead torsion}, Bull. Amer. Math. Soc. 
{\bf 72}  (1966), 358--426. 
\bibitem[Mo]{Mo} G.D. Mostow, {\it Strong rigidity of locally symmetric
 spaces}, Princeton Univ. Press and Univ. of Tokyo Press, 1973.
\bibitem[Mu1]{Mu1} W. M\"uller, {\it Analytic torsion and $R$-torsion for 
unimodular representations}, J. Amer. Math. Soc. {\bf 6} (1993), 721--753.
\bibitem[Pr]{Pr} G. Prasad, {\it Strong rigidity of $\Q$-rank 1 lattices,} 
Invent. math. {\bf 21} (1973), 255--286.
\bibitem[RS]{RS} D.B. Ray, I.M. Singer; {\it $R$-torsion and the Laplacian on 
Riemannian manifolds},  Advances in Math.  7, 145--210. (1971). 
\bibitem[Sh]{Sh} M.A. Shubin, {\it Pseudodifferential operators and spectral 
theory},  Second edition. Springer-Verlag, Berlin, 2001. 
\bibitem[Si]{Si} J. Silhan, {\it A real analog of Kostant's version of the 
Bott-Borel-Weil theorem}, J. of Lie theory {\bf 14} (2004), 481--499.
\bibitem[Vo]{Vo}  A. Voros, {\it Spectral functions, special functions and the
Selberg zeta function}, Commun. Math. Phys. {\bf 110} (1987), 439--465.
\bibitem[T1]{T1} W. Thurston, {\it Three-dimensional manifolds, Kleinian 
groups and hyperbolic geometry},  Bull. Amer. Math. Soc. {\bf 6}  (1982), 
no. 3, 357--381. 
\bibitem[T2]{T2} W. Thurston, {\it Three-dimensional geometry and topology}, 
Vol. 1.  Princeton Mathematical Series, 35. Princeton University Press, 
Princeton, NJ, 1997.
\bibitem[Wa1]{Wa1} N.R. Wallach, {\it On the Selberg trace formula in the case 
of compact quotient}, Bull. Amer. Math. Soc. {\bf 82} (1976), 171--195. 
\bibitem[Wo]{Wo} A. Wotzke, {\it Die Ruellsche Zetafunktion und die analytische
Torsion hyperbolischer Mannigfaltigkeiten}, Ph.D. thesis, Bonn, 2008, Bonner
Mathematische Schriften, Nr. 389. 
\bibitem[Zi]{Zi} B. Zimmermann, {\it A note on hyperbolic $3$-manifolds of the 
same volume},  Monatsh. Math.  {\bf 117},  (1994),  no. 1-2, 139--143. 
\end{thebibliography}
\end{document}